\documentclass[12pt]{amsart}
\usepackage{hyperref,cleveref,amssymb,mathtools}

\theoremstyle{plain}
\newtheorem{theorem}{Theorem}[section]
\newtheorem{proposition}[theorem]{Proposition}
\newtheorem{corollary}[theorem]{Corollary}
\newtheorem{lemma}[theorem]{Lemma}

\theoremstyle{definition}
\newtheorem{remark}[theorem]{Remark}
\newtheorem{question}[theorem]{Question}

\newtheorem{example}[theorem]{Example}
\newtheorem{definition}[theorem]{Definition}

\newcommand{\abs}[1]{\lvert#1\rvert}
\newcommand{\norm}[1]{\lVert#1\rVert}
\newcommand{\bigabs}[1]{\bigl\lvert#1\bigr\rvert}
\newcommand{\bignorm}[1]{\bigl\lVert#1\bigr\rVert}
\newcommand{\Bigabs}[1]{\Bigl\lvert#1\Bigr\rvert}
\newcommand{\biggabs}[1]{\biggl\lvert#1\biggr\rvert}
\newcommand{\Bignorm}[1]{\Bigl\lVert#1\Bigr\rVert}
\newcommand{\biggnorm}[1]{\biggl\lVert#1\biggr\rVert}

\newcommand{\term}[1]{{\textit{\textbf{#1}}}}   
\renewcommand{\mid}{\::\:}

\newcommand{\goeso}{\xrightarrow{\mathrm{o}}}	
\newcommand{\goesso}{\xrightarrow{\sigma\mathrm{o}}}	
\newcommand{\goesu}{\xrightarrow{\mathrm{u}}} 
\newcommand{\goesuo}{\xrightarrow{\mathrm{uo}}}	
\newcommand{\goesnorm}{\xrightarrow{\norm{\cdot}}} 

\newcommand{\s}{\sum_{k=1}^\infty}

\newcommand{\tos}{\prescript{{\rm o}}{}{\sum_{k=1}^\infty\,}}
\newcommand{\tsos}{\prescript{{\sigma\rm o}}{}{\sum_{k=1}^\infty\,}}
\newcommand{\tus}{\prescript{{\rm u}}{}{\sum_{k=1}^\infty\,}}
\newcommand{\tuos}{\prescript{{\rm uo}}{}{\sum_{k=1}^\infty\,}}

\DeclareSymbolFont{bbold}{U}{bbold}{m}{n}
\DeclareSymbolFontAlphabet{\mathbbold}{bbold}
\def\one{\mathbbold{1}}

\DeclareMathOperator{\Range}{Range}
\DeclareMathOperator{\Span}{span}
\DeclareMathOperator{\supp}{supp}
\DeclareMathOperator*{\Ave}{Ave}

\renewcommand{\le}{\leqslant}
\renewcommand{\ge}{\geqslant}

\begin{document}

\title[Bibasic sequences]
{Bibasic sequences\\ in Banach lattices}

\author{M.A. Taylor}
\author{V.G. Troitsky}
\address{Department of Mathematics, University of California,
  Berkeley, CA, 94720, United States.}
\email{mitchelltaylor@berkeley.edu}
\address{Department of Mathematical and Statistical Sciences,
         University of Alberta, Edmonton, AB, T6G\,2G1, Canada.}
\email{troitsky@ualberta.ca}

\thanks{The second author was supported by an NSERC grant.}

\date{\today}

\begin{abstract}
  Given a Schauder basic sequence $(x_k)$ in a Banach lattice, we say
  that $(x_k)$ is bibasic if the expansion of every vector in $[x_k]$
  converges not only in norm, but also in order. We prove that, in this
  definition, order convergence may be replaced with uniform
  convergence, with order boundedness of the partial sums, or with
  norm boundedness of finite suprema of the partial sums.

  The results in this paper extend and unify those from the pioneering
  paper \emph{Order Schauder bases in Banach lattices} by A.~Gumenchuk,
  O.~Karlova, and M.~Popov. In particular, we are able to characterize
  bibasic sequences in terms of the bibasis inequality, a result they
  obtained under certain additional assumptions.

  After establishing the aforementioned characterizations of bibasic
  sequences, we embark on a deeper study of their properties. We show,
  for example, that they are independent of ambient space, stable
  under small perturbations, and preserved under sequentially
  uniformly continuous norm isomorphic embeddings. After this we
  consider several special kinds of bibasic sequences, including
  permutable sequences, i.e., sequences for which every permutation
  is bibasic, and absolute sequences, i.e., sequences where
  expansions remain convergent after we replace every term with its
  modulus. We provide several equivalent characterizations of absolute
  sequences, showing how they relate to bibases and to further
  modifications of the basis inequality.

  We further consider bibasic sequences with unique order
  expansions. We show that this property does generally depend on
  ambient space, but not for the inclusion of $c_0$
  into~$\ell_\infty$. We also show that small perturbations of bibases
  with unique order expansions have unique order expansions, but this
  is not true if ``bibases'' is replaced with ``bibasic sequences''.

  In the final section, we consider uo-bibasic sequences, which are
  obtained by replacing order convergence with uo-convergence  in the
  definition of a bibasic sequence. We show that such sequences are
  very common.
\end{abstract}

\maketitle

\section{Preliminaries}

\subsection*{Schauder bases and decompositions.}
In this subsection, we collect notation and basic facts about Schauder
bases and decompositions. For details, we refer the reader to
\cite{Lindenstrauss:77,Singer:70,Singer:81}.  A sequence $(x_k)$ in a
Banach space $X$ is said to be a \term{(Schauder) basis} of $X$ if
every vector $x$ in $X$ admits a unique decomposition
$x=\s\alpha_k x_k$, where the series converges in norm. For each~$n$,
we define the $n$-th \term{basis projection} $P_n\colon X\to X$ via
\begin{math}
  P_n\bigl(\s\alpha_k x_k)=\sum_{k=1}^n\alpha_k x_k.
\end{math}
We define the $n$-th \term{coordinate functional} $x_n^*$ via
\begin{math}
  x^*_n\bigl(\s\alpha_k x_k)=\alpha_n.
\end{math}
It is known that the $P_n$'s are uniformly bounded; the number
$K=\sup_n\norm{P_n}$ is called the \term{basis constant} of~$(x_k)$. A
sequence $(x_k)$ in $X$ is called a \term{(Schauder) basic sequence}
if it is a basis for its closed linear span~$[x_k]$; in this case the
$P_n$'s and $x_n^*$'s are defined on~$[x_k]$. It is a standard fact
that a sequence $(x_k)$ of non-zero vectors in $X$ is basic iff there
exists $K\ge 1$ such that
\begin{equation}\label{basis-constant}
  \Bignorm{\sum_{k=1}^n\alpha_kx_k}\le K\Bignorm{\sum_{k=1}^m\alpha_kx_k}
\end{equation}
for every $n\le m$ and all scalars $\alpha_1,\dots,\alpha_m$; the least
value of the constant $K$ is the basis constant of~$(x_k)$.

More generally, suppose that $(X_k)$ is a sequence of closed non-zero
subspaces of a Banach space~$X$; let $[X_k]$ be the closed linear span
of $\bigcup_{k=1}^\infty X_k$. We say that $(X_k)$ is a
\term{(Schauder) decomposition} of $[X_k]$ if every $x$ in $[X_k]$
admits a unique expansion $x=\s x_k$, where $x_k\in X_k$ for each $k$
and the series converges in norm. As before, we define the
\term{canonical projections} $P_n\colon[X_k]\to[X_k]$ via
$P_nx=\sum_{k=1}^nx_k$. These projections are uniformly bounded;
moreover, a sequence $(X_k)$ of closed non-zero subspaces of $X$ is a
Schauder decomposition iff there exists a constant $K\ge 1$ such that
\begin{equation}
  \label{decomp-basis-constant}
    \Bignorm{\sum_{k=1}^nx_k}\le K\Bignorm{\sum_{k=1}^mx_k}
\end{equation}
whenever $n\le m$ and $x_k\in X_k$ for all $k=1,\dots,m$; see, e.g.,
Theorem~15.5 in \cite[p.~502]{Singer:81}. Clearly, every basic
sequence $(x_k)$ induces a Schauder decomposition with
$X_k=\Span x_k$. We refer the reader to \cite[\S 15]{Singer:81} or
\cite[1.g]{Lindenstrauss:77} for further information on Schauder
decompositions. Note that unlike \cite{Singer:81,Lindenstrauss:77}, we
do not assume that $[X_k]=X$. The reason is that in this paper $X$
will generally be a Banach lattice, but we will not require $[X_k]$ to
form a sublattice.

For our purposes, it is important to note that \eqref{basis-constant}
and \eqref{decomp-basis-constant} may be re-written as follows:
\begin{displaymath}
  \bigvee_{n=1}^m\Bignorm{\sum_{k=1}^n\alpha_kx_k}
  \le K\Bignorm{\sum_{k=1}^m\alpha_kx_k}
  \quad\mbox{and}\quad
  \bigvee_{n=1}^m\Bignorm{\sum_{k=1}^nx_k}
  \le K\Bignorm{\sum_{k=1}^mx_k}.
\end{displaymath}

\subsection*{Uniform and order convergence}
Let $X$ be an Archimedean vector lattice. A net $(x_\alpha)$
\term{converges uniformly} to~$x$, denoted $x_\alpha\goesu x$, if
there exists $e\in X_+$ such that for every $\varepsilon>0$ there
exists $\alpha_0$ such that $\abs{x_\alpha-x}\le\varepsilon e$
whenever $\alpha\ge\alpha_0$. We say that $(x_\alpha)$ \term{converges
  in order} to $x$ and write $x_\alpha\goeso x$ if there exists a net
$(u_\gamma)$ (which may have a different index set) such that
$u_\gamma\downarrow 0$ and for every $\gamma$ there exists $\alpha_0$
such that $\abs{x_\alpha-x}\le u_\gamma$ whenever
$\alpha\ge\alpha_0$. A sequence $(x_n)$ is said to
\term{$\sigma$-order converge} to~$x$, written $x_n\goesso x$, if
there exists a sequence $(u_n)$ such that $u_n\downarrow 0$ and
$\abs{x_n-x}\le u_n$ for every~$n$. In some of the literature,
$\sigma$-order convergence is called ``order convergence for
sequences''. It is easy to see that
\begin{displaymath}
  x_n\goesu x\quad\Rightarrow\quad
  x_n\goesso x\quad\Rightarrow\quad
  x_n\goeso x.
\end{displaymath}
Although order convergence and $\sigma$-order convergence disagree in
general, they agree for sequences in $\sigma$-order complete vector
lattices.  Clearly, uniform convergence implies order convergence; in
Banach lattices, uniform convergence implies norm convergence.

\begin{lemma}\label{nconv-subseq}
  Let $(x_k)$ be a sequence in a Banach lattice $X$ such that the
  series $\s\norm{x_k}$ converges. Then $x_k\goesu 0$ and the series
  $\s x_k$ converges both in norm and uniformly. In
  particular, every norm convergent sequence in $X$ has a subsequence
  which converges uniformly and, therefore, in order.
\end{lemma}

\begin{proof}
  Find a sequence $(\lambda_k)$ such that
  $1\le\lambda_k\uparrow\infty$ and
  $\s\lambda_k\norm{x_k}<\infty$. Put $u=\s\lambda_k\abs{x_k}$. Then
  $\abs{x_k}\le\frac{1}{\lambda_k}u$ for every~$k$, hence
  $x_k\goesu 0$. Clearly, the series $\s x_k$ converges in norm; let
  $x$ be the sum. Then
  \begin{displaymath}
    \Bigabs{x-\sum_{k=1}^nx_k}\le\sum_{k=n+1}^\infty\abs{x_k}
    \le\frac{1}{\lambda_n}\sum_{k=n+1}^\infty\lambda_k\abs{x_k}
    \le\frac{1}{\lambda_n}u,
  \end{displaymath}
  so that the series converges uniformly.
\end{proof}


A Banach lattice $X$ is said to be \term{order continuous} if
$x_\alpha\goeso 0$ implies $x_\alpha\goesnorm 0$ for every net
$(x_\alpha)$ in~$X$. It follows from the Meyer-Nieberg Theorem that a
Banach lattice is order continuous iff $x_n\goeso 0$ implies
$x_n\goesnorm 0$ for every sequence $(x_n)$ in~$X$. We say that $X$ is
\term{$\sigma$-order continuous} if $x_n\goesso 0$ implies
$x_n\goesnorm 0$. It can be easily seen that a Banach lattice is order
continuous iff uniform convergence agrees with order convergence on
nets iff uniform convergence agrees with order convergence on
sequences; a Banach lattice is $\sigma$-order continuous iff uniform
convergence agrees with $\sigma$-order convergence on sequences; see,
e.g.~\cite{Bedingfield:80}.
  
We next provide two standard examples to illustrate the varied
relationships between uniform, norm, and order convergence.

\begin{example}
  Let $X=L_p(\mu)$ where $\mu$ is a measure and $1\le p<\infty$. Then
  $X$ is order continuous and a sequence $(f_k)$ converges in order to
  $f$ iff $(f_k)$ is order bounded and $(f_k)$ converges to $f$ almost
  everywhere (a.e.).
\end{example}

\begin{example}\label{non-o-conv}
  Let $X=C[0,1]$. It is easy to see that uniform convergence agrees
  with norm convergence. For each $k\in\mathbb N$, let $f_k\in X$ be
  such that $f_k(0)=1$, $f_k$ is linear on $[0,\frac{1}{2^k}]$, and
  $f_k$ vanishes on $[\frac{1}{2^k},1]$. Then $f_k\downarrow 0$, hence
  $f_k\goeso 0$. However, $(f_k)$ does not converge to zero in norm.
\end{example}

We will write $\s x_k$, $\tos x_k$, $\tsos x_k$, and $\tus x_k$ for
the norm, order, $\sigma$-order, and uniformly convergent series,
respectively. It follows from the last part of
Lemma~\ref{nconv-subseq} that if both $\s x_k$ and $\tos x_k$ converge
then they have the same sum. Replacing norm convergence in the
definition of a Schauder basis with uniform, order, or $\sigma$-order
convergence, one obtains the concepts of a uniform, order, and
$\sigma$-order basis, respectively, in a vector lattice.  Note that
the concepts of an order and a $\sigma$-order basis agree in
$\sigma$-order complete vector lattices; the concepts of a
$\sigma$-order and a uniform basis agree in $\sigma$-order continuous
Banach lattices. Although such bases will not be the focus of this
paper, we provide some simple examples that will be used later on.

\begin{example}
  Let $X=\ell_p$ with $1\le p<\infty$ or $X=c_0$. The standard unit
  vector sequence $(e_k)$ is a Schauder basis, an order basis, and a
  uniform basis. Note that $(e_k)$ is an order basis in~$\ell_\infty$,
  though it is neither a Schauder basis nor a uniform basis.
\end{example}

\begin{example}\label{C01-Sch}
  Let $X=C[0,1]$ and consider the Schauder system $(x_k)$ in $C[0,1]$
  as described in, e.g., \cite[p.~3]{Lindenstrauss:77}.  Since uniform
  and norm convergence agree in $C[0,1]$, $(x_k)$ is a uniform
  basis. However, it is not an order basis. Indeed, it can be easily
  verified that there is a sequence of coefficients $(\alpha_k)$ such
  that the sequence $(f_k)$ in Example~\ref{non-o-conv} is a tail of
  the sequence of partial sums for the series $\s\alpha_kx_k$. It
  follows that this series converges in order to zero. Hence, zero has
  non-unique order expansions.
\end{example}

\begin{example}\label{dual-summ}
  Let $X=\ell_1$, put $x_1=e_1$ and $x_k=-e_{k-1}+e_k$ when $k>1$. It
  is easy to see that $(x_k)$ is a Schauder basis of~$\ell_1$. We
  claim that it is neither a uniform basis nor an order basis.
  Consider the series $x=\sum_{k=1}^\infty\frac{x_k}{k}$. This series
  converges in norm, but its partial sums are not order bounded, hence
  it fails to converge uniformly or in order. It follows that $(x_k)$
  is neither an order basis nor a uniform basis because otherwise the
  uniform and the order expansion of $x$ would have to agree with its
  norm expansion.
\end{example}

\subsection*{Martingale inequalities}

We recall two classical martingale inequalities. Let $(f_k)$ be a
martingale in $L_1(P)$ for some probability measure~$P$; let $(d_k)$
be its difference sequence, i.e., $f_n=\sum_{k=1}^nd_k$ for
every~$n$. Put $f^*=\sup_k\abs{f_k}$ and
$S(f)=\bigl(\sum_{k=1}^\infty d_k^2\bigr)^\frac12$; these functions
are computed pointwise and are called the maximal and the square
function of~$(f_k)$, respectively. Let $1\le p<\infty$.  It is easy to
see that $\norm{f_k}_{L_p}\le\norm{f_{k+1}}_{L_p}$.  If $(f_k)$ is
norm bounded in $L_p(P)$ for some $1<p<\infty$, then Doob's inequality
asserts that
\begin{equation}\label{Doob}
  \norm{f^*}_{L_p}\le q\sup_k\norm{f_k}_{L_p}
\end{equation}
where $q=p^*$; see, e.g., Theorem~26.3 in~\cite{Jacod:03}.
Furthermore, Burkholder-Gundy-Davis inequality asserts that for every
$1\le p<\infty$,
\begin{equation}\label{BGD}
  C_1\norm{S(f)}_{L_p}\le\norm{f^*}_{L_p}\le C_2\norm{S(f)}_{L_p},
\end{equation}
where $C_1$ and $C_2$ depend only on~$p$; see, e.g., \cite{Davis:70}.

\section{The bibasis theorem}

The present paper will center around the following theorem:

\begin{theorem}\label{bib}
  Let $(x_k)$ be a Schauder basic sequence in a Banach
  lattice~$X$. TFAE:
  \begin{enumerate}
  \item\label{bib-uconv} $x=\s\alpha_kx_k$ implies
    $x=\tus\alpha_kx_k$ for every sequence~$(\alpha_k)$;
  \item\label{bib-soconv} $x=\s\alpha_kx_k$ implies
    $x=\tsos\alpha_kx_k$ for every sequence~$(\alpha_k)$;
  \item\label{bib-oconv} $x=\s\alpha_kx_k$ implies
    $x=\tos\alpha_kx_k$ for every sequence~$(\alpha_k)$;
  \item\label{bib-obdd} If $\sum\limits_{k=1}^\infty\alpha_kx_k$
    converges then its partial sums $\sum\limits_{k=1}^n\alpha_kx_k$
    are order bounded;
  \item\label{bib-nbdd} If $\sum\limits_{k=1}^\infty\alpha_kx_k$
    converges then the sequence
    \begin{math}
     \Bigl(\bigvee\limits_{n=1}^m\Bigabs{\sum\limits_{k=1}^n\alpha_kx_k}
     \Bigr)_{m=1}^\infty
   \end{math}
    is norm bounded;
  \item\label{bib-ineq}
    \begin{math}
      \exists M\ge 1\quad\forall m\in\mathbb
      N\quad\forall\alpha_1,\dots,\alpha_m\in\mathbb R
    \end{math}
    \begin{displaymath}
      \biggnorm{\bigvee\limits_{n=1}^m\Bigabs{\sum\limits_{k=1}^n\alpha_kx_k}}
      \le M\Bignorm{\sum_{k=1}^m\alpha_kx_k}.
    \end{displaymath}
  \end{enumerate}
\end{theorem}

A basic sequence $(x_k)$ satisfying any (and, therefore, all) of the
equivalent conditions in the theorem will be referred to as a
\term{bibasic sequence}. If, in addition, $[x_k]=X$ we will call it a
\term{bibasis}. The condition in \eqref{bib-ineq} will be referred to
as the \term{bibasis inequality}, and the least value of $M$ for which
this inequality is satisfied will be called the \term{bibasis
  constant} of~$(x_k)$.  It is easy to see that
\eqref{bib-uconv}$\Rightarrow$%
\eqref{bib-soconv}$\Rightarrow$%
\eqref{bib-oconv}$\Rightarrow$%
\eqref{bib-obdd}$\Rightarrow$%
\eqref{bib-nbdd}.
Instead of proving the rest of the theorem
directly, we will deduce it as an immediate corollary of a more
general fact in Section~\ref{sec:bidec}. Before that, we present a few
remarks and corollaries.

\begin{remark}\label{bib-bib-ineq}
  In the theorem, we assume that $(x_k)$ is a Schauder basic
  sequence. However, a sequence of non-zero vectors which satisfies
  the bibasis inequality also satisfies the basis
  inequality~\eqref{basis-constant}, hence is automatically a Schauder
  basic sequence. Thus, a sequence in $X\setminus\{0\}$ is bibasic iff
  it satisfies the bibasis inequality.
\end{remark}

We emphasize that only~$X$, not~$[x_k]$, is assumed to be a
lattice. It follows from~\eqref{bib-nbdd} that the concept of a
bibasic sequence does not depend on the ambient space $X$:

\begin{corollary}\label{ambient}
  Let $Y$ be a closed sublattice of a Banach lattice~$X$, and $(x_k)$
  a sequence in~$Y$. Then $(x_k)$ is bibasic in $Y$ iff it is bibasic
  in~$X$. 
\end{corollary}

The following is immediate.

\begin{corollary}\label{Sch-unif}
  $(x_k)$ is a bibasis iff it is both a Schauder basis and a uniform
  basis. If $X$ is $\sigma$-order continuous then $(x_k)$ is a bibasis
  iff it is both a Schauder basis and a $\sigma$-order basis.
\end{corollary}





\begin{question}
  It is not known whether every uniform basis of a Banach lattice is
  automatically a Schauder basis; cf \cite[Problem 1.3]{Gumenchuk:15}.
\end{question}

\begin{example}\label{nonbib-l1}
  It is now somewhat easier to see that the sequence $(x_k)$ in
  Example~\ref{dual-summ} is not  a uniform
  basis. Indeed, taking $\alpha_k=1$ as $k=1,\dots,m$, it is clear
  that the bibasis inequality fails, hence $(x_k)$ is not a bibasis
  and, therefore, is not a uniform basis by
  \Cref{Sch-unif}.
\end{example}

\begin{remark}\label{AM-basic-bib}
  It follows from~\eqref{bib-nbdd} or~\eqref{bib-ineq} that every
  basic sequence in an AM-space is bibasic. In particular, the
  Schauder system of $C[0,1]$ in Example~\ref{C01-Sch} is a bibasis,
  even though it is not an order basis.
\end{remark}

\begin{question}
  Let $X$ be a Banach lattice and suppose that every basic sequence in
  $X$ is bibasic. Does this imply that $X$ is lattice isomorphic to an
  AM-space?
\end{question}

\begin{remark}
  The concepts of a bibasis and a bibasic sequence were originally
  introduced in \cite{Gumenchuk:15}. Formally speaking, the definition
  in \cite{Gumenchuk:15} is slightly different: they defined a bibasis
  as a sequence which is both a Schauder basis and a $\sigma$-order
  basis. For example, the Schauder system of $C[0,1]$ is a bibasis in
  our sense, but not in the sense of \cite{Gumenchuk:15}. However, in
  \cite{Gumenchuk:15} the authors only consider $\sigma$-order
  continuous spaces, and in this case the two definitions agree by
  Corollary~\ref{Sch-unif}, so that all the results of
  \cite{Gumenchuk:15} remain valid for our definition.

  In \cite{Gumenchuk:15}, it was proved that if $X$ is $\sigma$-order
  continuous then every bibasis satisfies the bibasis inequality,
  which corresponds to the implication
  \eqref{bib-soconv}$\Rightarrow$\eqref{bib-ineq} of
  Theorem~\ref{bib}. They also proved
  \eqref{bib-ineq}$\Rightarrow$\eqref{bib-soconv} under certain
  additional assumptions. Thus, our Theorem~\ref{bib} improves the
  results of \cite{Gumenchuk:15}: we make no assumptions on~$X$, we
  add conditions \eqref{bib-uconv}, \eqref{bib-oconv},
  \eqref{bib-obdd}, and \eqref{bib-nbdd}, and our proof is shorter.
\end{remark}

\begin{question}
  We do not know whether the definition of a bibasis in
  \cite{Gumenchuk:15} implies our definition in an arbitrary Banach
  lattice. Equivalently, if $(x_k)$ is a Schauder basis and a
  $\sigma$-order basis, do the coefficients in the norm and in the
  order expansions of the same vector agree?
\end{question}

\begin{example}\label{Haar}
  Let $X=L_p[0,1]$; let $(h_k)$ be the Haar system in $L_1[0,1]$ as
  described in \cite[p.~3]{Lindenstrauss:77}. In particular, $(h_k)$
  is a monotone Schauder basis in~$X$.  It was shown in
  \cite{Gumenchuk:15} that the Haar system $(h_k)$ fails to be a
  bibasis in $L_1[0,1]$. It was also shown there that $(h_k)$ is a
  bibasis when $1<p<\infty$, and its bibasis constant satisfies
  $M_p\ge\bigl(1+\frac{1}{2^p-2}\bigr)^\frac1p$. We present an
  alternative approach to this problem. Let $1<p<\infty$. Fix
  $\alpha_1,\dots,\alpha_m$, put $x_n=\sum_{k=1}^n\alpha_kh_k$ as
  $n=1,\dots,m$. It is easy to see that $(x_n)_{n=1}^m$ is a
  martingale. By Doob's Inequality~\eqref{Doob},
  \begin{displaymath}
    \Bignorm{\bigvee_{n=1}^m\abs{x_n}}_{L_p}
    \le q\norm{x_m}_{L_p}.
  \end{displaymath}
  This yields that $(h_k)$ satisfies the bibasis inequality with
  $M_p=q$. Furthermore, it was shown in~\cite[p.~15]{Burkholder:91}
  that the constant $q$ is sharp in Doob's inequality even for dyadic
  martingales. Since dyadic martingales are of the form
  $\bigl(\sum_{k=1}^n\alpha_kh_k\bigr)_{n=1}^\infty$, it follows that
  the bibasis constant of $(h_k)$ in $L_p[0,1]$ equals~$q$. This
  argument also shows that every martingale difference sequence in
  $L_p(P)$ with $1<p<\infty$ is bibasic.
\end{example}

We finish this section with a comment about ambient space. As observed
in Corollary~\ref{ambient}, the concept of a bibasic sequence does not
depend on the ambient space because, according to
parts~\eqref{bib-nbdd} and~\eqref{bib-ineq} of Theorem~\ref{bib},
bibasic sequences may be characterized in terms of the norm and
lattice operations only, and the latter do not depend on the ambient
space. Parts~\eqref{bib-soconv}, \eqref{bib-oconv},
and~\eqref{bib-obdd} characterize bibasic sequences in terms of order
convergence, $\sigma$-order convergence, and order boundedness; these
three concepts may depend on ambient space. For example, the standard
unit vector basis $(e_k)$ of $c_0$ is neither order convergent nor
even order bounded, yet it is order null when viewed as a sequence
in~$\ell_\infty$. So it is somewhat surprising that order convergence
and order boundedness of bibasic expansions in~\eqref{bib-soconv},
\eqref{bib-oconv}, and~\eqref{bib-obdd} do not depend on the ambient
space.  This leaves~\eqref{bib-uconv}: does uniform convergence depend
on the ambient space? It is easy to see that uniform convergence in a
sublattice implies uniform convergence in the entire space; however,
the converse is false in the category of vector lattices. For example,
the sequence $\frac{1}{k}e_k$ is uniformly null in~$\ell_\infty$, but
not in~$\ell_1$. Nevertheless, the following proposition shows that
uniform convergence does not depend on the ambient space in the
category of Banach lattices.

\begin{proposition}\label{uconv-stable}
  Let $Y$ be a closed sublattice of a Banach lattice $X$ and $(x_k)$ a
  sequence in~$Y$. Then $x_k\goesu 0$ in $Y$ iff $x_k\goesu 0$ in~$X$.
\end{proposition}

\begin{proof}
  The forward implication is trivial. Suppose $x_k\goesu 0$
  in~$X$. Find $e\in X_+$ such that $(x_k)$ converges to zero
  uniformly relative to~$e$. WLOG, scaling everything, we may assume
  that $\norm{e}=1$. For every $n$ there exists $k_n$ such that
  $\abs{x_k}\le \frac{1}{n^3}e$ for all $k\ge k_n$. WLOG, $(k_n)$ is
  an increasing sequence. For every~$n$, put
  \begin{math}
    v_n=\bigvee_{k=k_n}^{k_{n+1}-1}\abs{x_k}.
  \end{math}
  Then
  $v_n\le\frac{1}{n^3}e$ and, therefore,
  $\norm{v_n}\le\frac{1}{n^3}$. It follows that the series
  $w:=\sum_{n=1}^\infty nv_n$ converges and $w\in Y$. It is left to
  show that $(x_k)$ converges to zero uniformly relative to~$w$. Let
  $n\in\mathbb N$. Take any $k\ge k_n$. Find $m\ge n$ such that
  $k_m\le k< k_{m+1}$. Then $\abs{x_k}\le
  v_m\le\frac{1}{m}w\le\frac{1}{n}w$. 
\end{proof}

The preceding proposition fails for nets. Consider the double sequence
$x_{n,m}=\frac1n e_m$ in~$\ell_\infty$. It follows from
$x_{n,m}\le\frac1n\one$ that $x_{n,m}\goesu 0$
in~$\ell_\infty$. However, viewed as a net in~$c_0$, its tails are not
order bounded, hence it fails to converge uniformly to zero.

It is also worth mentioning that the preceding proposition remains
valid for uniformly closed sublattices of uniformly complete vector
lattices, with essentially the same proof.

\section{Bidecompositions}
\label{sec:bidec}

From now on, when possible, we will work in the language of
decompositions. In particular, the results apply to basic
sequences. However, we find the language of decompositions more
natural and clear for our purposes.

\begin{theorem}\label{decomp}
  Let $X$ be a Banach lattice and $(X_k)\subseteq X$ a Schauder
  decomposition of~$[X_k]$. Let $P_n\colon[X_k]\to[X_k]$ be the $n$-th
  canonical projection. TFAE:
  \begin{enumerate}
  \item\label{decomp-uconv} For all $x\in[X_k]$, $P_nx\goesu x$;
  \item\label{decomp-soconv} For all $x\in[X_k]$, $P_nx\goesso x$;
  \item\label{decomp-oconv} For all $x\in[X_k]$, $P_nx\goeso x$;
  \item\label{decomp-obdd} For all $x\in[X_k]$, $(P_nx)$ is order
    bounded in~$X$;
  \item\label{decomp-nbdd} For all $x\in[X_k]$,
    \begin{math}
      \bigl(\bigvee_{n=1}^m\abs{P_nx}\bigr)_{m=1}^\infty
    \end{math}
    is norm bounded;
  \item\label{decomp-ineq} There exists $M\ge 1$ such that for any
    $m\in\mathbb N$ and any $x_1\in X_1$, \dots, $x_m\in X_m$ one has
    \begin{equation}\label{bidec-ineq}
      \Bignorm{\bigvee\limits_{n=1}^m\Bigabs{\sum\limits_{k=1}^nx_k}}
      \le M\Bignorm{\sum_{k=1}^mx_k}.
    \end{equation}
  \end{enumerate}
\end{theorem}

\begin{proof}
  \eqref{decomp-uconv}$\Rightarrow$%
  \eqref{decomp-soconv}$\Rightarrow$%
  \eqref{decomp-oconv}$\Rightarrow$%
  \eqref{decomp-obdd}$\Rightarrow$%
  \eqref{decomp-nbdd} is straightforward.
  
  \eqref{decomp-nbdd}$\Rightarrow$\eqref{decomp-ineq}: For every
  $i\in\mathbb N$, put
  \begin{displaymath}
    F_i=\biggl\{x\in[X_k]\mid\sup\limits_m
    \Bignorm{\bigvee_{n=1}^m\abs{P_nx}}\le i\biggr\}
    =\bigcap_{m=1}^\infty\biggl\{x\in[X_k]\mid
    \Bignorm{\bigvee_{n=1}^m\abs{P_nx}}\le i\biggr\}.
  \end{displaymath}
  Continuity of the canonical projections and lattice operations
  yields that each $F_i$ is closed. It follows from
  \eqref{decomp-nbdd} that $\bigcup_{i=1}^\infty F_i=[X_k]$. By Baire
  Category theorem, there exists $i_0$ such that $F_{i_0}$ has
  non-empty interior relative to~$[X_k]$. That is, there exists
  $x_0\in F_{i_0}$ and $\varepsilon>0$ such that $x\in F_{i_0}$
  whenever $x\in[X_k]$ and $\norm{x-x_0}\le\varepsilon$. Let
  $x\in[X_k]$ with $\norm{x}\le 1$. For each $n\in\mathbb N$, the
  triangle inequality yields
  \begin{math}
    \varepsilon\abs{P_nx}\le\abs{P_nx_0}+\abs{P_n(x_0+\varepsilon x)}.
  \end{math}
  It follows that
  \begin{displaymath}
    \varepsilon\bigvee_{n=1}^m\abs{P_nx}\le
    \bigvee_{n=1}^m\abs{P_nx_0}+
    \bigvee_{n=1}^m\abs{P_n(x_0+\varepsilon x)}
  \end{displaymath}
  for each $m\in\mathbb N$, so that 
  \begin{math}
    \varepsilon\Bignorm{\bigvee_{n=1}^m\abs{P_nx}}\le 2i_0.
  \end{math}
  This yields
  \begin{displaymath}
    \Bignorm{\bigvee_{n=1}^m\abs{P_nx}}
    \le\frac{2i_0}{\varepsilon}\norm{x}
  \end{displaymath}
   for all $x\in[X_k]$ and $m\in\mathbb N$. Now given $m$ and
  $x_1\in X_1$, \dots, $x_m\in X_m$, define $x=\sum_{k=1}^mx_k$ to get
  \eqref{decomp-ineq} with $M=\frac{2i_0}{\varepsilon}$.

  \eqref{decomp-ineq}$\Rightarrow$\eqref{decomp-uconv}:
  Let $x\in[X_k]$ and let $(x_k)$ be the unique sequence such that
  $x_k\in X_k$ for every $k$ and $P_nx=\sum_{k=1}^nx_k\goesnorm
  x$. Then there is a subsequence $(P_{n_m}x)$ such that
  $P_{n_m}x\goesu x$. WLOG, passing to a further
  subsequence and using that $(P_nx)$ is norm Cauchy, we may assume
  that
  \begin{math}
    \bignorm{\sum_{k=n_m+1}^ix_k}<\frac{1}{2^m}
  \end{math}
  whenever $i>n_m$.

  For every $m\in\mathbb N$, define
  \begin{displaymath}
    u_m=\bigvee_{i=n_m+1}^{n_{m+1}}\Bigabs{\sum_{k=n_m+1}^ix_k}.
  \end{displaymath}
  Applying \eqref{decomp-ineq} to the sequence
  \begin{math}
    (0,\dots,0,x_{n_m+1},x_{n_m+2},\dots, x_{n_{m+1}})
  \end{math}
  with $n_m$ zeros at the beginning yields
  \begin{displaymath}
    \norm{u_m}\le M\Bignorm{\sum_{k=n_m+1}^{n_{m+1}}x_k}<\frac{M}{2^m}.
  \end{displaymath}
  Define $u=\sum_{m=1}^\infty mu_m$; it follows from
  $u_m\le\frac{u}{m}$ that $u_m\goesu 0$. 

  Therefore, $\abs{P_{n_m}x-x}+u_m\goesu 0$. It follows that there is
  a vector $e>0$ with the property that for any $\varepsilon>0$ there
  exists $m_0$ such that $\abs{P_{n_m}x-x}+u_m\le\varepsilon e$
  whenever $m\ge m_0$. Fix $\varepsilon>0$, and find the required~$m_0$. Let
  $i\in\mathbb N$ with $i>n_{m_0}$. Then we can find $m\ge m_0$ such
  that $n_m<i\le n_{m+1}$, so that
  \begin{multline*}
    \abs{P_ix-x}\le\abs{P_{n_m}x-x}+\abs{P_ix-P_{n_m}x}\\
    =\abs{P_{n_m}x-x}+\Bigabs{\sum_{k=n_m+1}^ix_k}
    \le\abs{P_{n_m}x-x}+u_m\le\varepsilon e.
  \end{multline*}
  This shows that $P_ix\goesu x$.
\end{proof}

A Schauder decomposition $(X_k)$ satisfying the equivalent conditions
of Theorem~\ref{decomp} will be referred to as a
\term{bidecomposition}, and the least value of $M$ in
\eqref{decomp-ineq} will be called the \term{bidecomposition constant}
of~$(X_k)$. Note that each $X_k$ (as well as $[X_k]$) is a closed
subspace of $X$ which need not be a sublattice.

As in Corollary~\ref{ambient}, it follows immediately from
\eqref{decomp-nbdd} that the definition of a bidecomposition does not
depend on ambient space. In particular, $(X_k)$ is a bidecomposition
in $X$ iff it is a bidecomposition in~$X^{**}$.

The following is an analogue of Remark~\ref{bib-bib-ineq}:

\begin{corollary}
  Let $(X_k)$ be a sequence of closed non-zero subspaces of a Banach
  lattice~$X$. Then $(X_k)$ is a bidecomposition iff it
  satisfies~\eqref{decomp-ineq}.
\end{corollary}

Clearly, every bibasic sequence $(x_k)$ induces a bidecomposition
$(X_k)$ with $X_k=\Span x_k$, hence Theorem~\ref{bib} is a special
case of Theorem~\ref{decomp}. Furthermore, let $(X_k)$ be a
bidecomposition; for each~$k$, pick a non-zero vector $x_k\in X_k$.
By Theorem~\ref{decomp}\eqref{decomp-ineq}, the resulting basic
sequence $(x_k)$ satisfies the bibasis inequality, hence is
bibasic. To show a partial converse, we will use the following known fact; see
Theorems~15.21(b) and~15.22(4) in~\cite{Singer:81}, pp.~543 and~546,
respectively. 

\begin{theorem}[\cite{Singer:81}]\label{Singer}
  Let $(X_k)$ be a sequence of closed non-zero subspaces of a Banach
  space~$X$. Suppose that every sequence $(x_k)$ satisfying
  $0\ne x_k\in X_k$ is Schauder basic. Then there exists
  $N\in\mathbb N$ such that the sequence $(X_k)_{k\ge N}$ is a
  Schauder decomposition. If $\dim X_k<\infty$ for every $k$ then one
  may choose $N=1$.
\end{theorem}

We will now prove a similar fact for bidecompositions and bibasic
sequences.

\begin{theorem}\label{bib-decomp}
  Let $(X_k)$ be a sequence of closed non-zero subspaces of a Banach
  lattice~$X$. Suppose that every sequence $(x_k)$ satisfying
  $0\ne x_k\in X_k$ is bibasic. Then there exists
  $N\in\mathbb N$ such that the sequence $(X_k)_{k\ge N}$ is a
  bidecomposition. Moreover, if the sequence $(X_k)$ is a Schauder
  decomposition or if $\dim X_k<\infty$ for every $k$ then
  one may choose $N=1$.
\end{theorem}

\begin{proof}
  By Theorem~\ref{Singer}, there exists $N\in\mathbb N$ such that the
  sequence $(X_k)_{k\ge N}$ is a Schauder decomposition (in the case
  when $(X_k)$ is already a Schauder decomposition or when
  $\dim X_k<\infty$ for every~$k$, take $N=1$). It suffices to show
  that the sequence $(X_k)_{k\ge N}$ satisfies~\eqref{decomp-uconv} in
  Theorem~\ref{decomp}. Let $x\in[X_k]_{k\ge N}$. For every $k\ge N$,
  there exists a unique $x_k\in X_k$ such that
  $x=\sum_{k=N}^\infty x_k$. For every $k\in\mathbb N$, if $k\ge N$
  and $x_k\ne 0$ then put $y_k=x_k$ and $\alpha_k=1$; if $k<N$ or
  $x_k=0$, put $\alpha_k=0$ and let $y_k$ be an arbitrary non-zero
  element of~$X_k$. By assumption,
  $(y_k)$ is bibasic. Then
  $x=\sum_{k=N}^\infty x_k=\sum_{k=1}^\infty\alpha_ky_k$ and
  Theorem~\ref{bib}\eqref{bib-uconv} guarantees that the series
  converges uniformly.
\end{proof}



\section{Stability of bibasic sequences under perturbations}

It was proved in Theorem~3.1 of~\cite{Gumenchuk:15} that if $X$ is
$\sigma$-order continuous then every block sequence of a bibasic
sequence is again bibasic, and the bibasis constant of the block
sequence does not exceed that of the original sequence. In particular,
every subsequence of a bibasic sequence is again bibasic. This result
now follows immediately from the bibasis inequality in
Theorem~\ref{bib}; moreover, if one uses our definition of a bibasic
sequence then the $\sigma$-order continuity assumption
is not needed:


\begin{corollary}\label{blocks}
  Let $(x_k)$ be a bibasic sequence in a Banach lattice. Then every
  block sequence of $(x_k)$ is bibasic with a bibasis constant that
  does not exceed that of~$(x_k)$. Similarly, every blocking of a
  bidecomposition is a bidecomposition.
\end{corollary}

The following result improves Theorem~3.2 of~\cite{Gumenchuk:15}: we
remove the assumption that the space is $\sigma$-order continuous and
we add an estimate on the bibasis constant. Essentially, we show that
a small perturbation of a bibasic sequence causes a small perturbation of the
bibasis constant.

\begin{theorem}\label{biPSP}
  Let $(x_k)$ be a bibasic sequence in a Banach lattice $X$ with basis
  constant $K$ and bibasis constant~$M$. Let $(y_k)$ be a sequence in
  $X$ with
  \begin{displaymath}
    2K\sum_{k=1}^\infty\frac{\norm{x_k-y_k}}{\norm{x_k}}=:\theta<1.
  \end{displaymath}
  Then $(y_k)$ is bibasic with bibasis constant at most
  \begin{math}
    \frac{M+\theta}{1-\theta}.
  \end{math}
\end{theorem}

\begin{proof}
  Fix scalars $\alpha_1,\dots,\alpha_m$; put
  $x=\sum_{k=1}^m\alpha_kx_k$ and $y=\sum_{k=1}^m\alpha_ky_k$. Note
  that $\abs{\alpha_k}\le\frac{2K\norm{x}}{\norm{x_k}}$ as
  $k=1,\dots,m$. Then
  \begin{displaymath}
    \norm{x-y}\le\sum_{k=1}^m\abs{\alpha_k}\norm{x_k-y_k}
    \le 2K\norm{x}\sum_{k=1}^m\frac{\norm{x_k-y_k}}{\norm{x_k}}
    \le\theta\norm{x}.
  \end{displaymath}
  This implies that
  \begin{math}
    \norm{x}\le\norm{x-y}+\norm{y}
    \le\theta\norm{x}+\norm{y},
  \end{math}
  so that $\norm{x}\le\frac{\norm{y}}{1-\theta}$.
  Define $u:=\sum_{k=1}^\infty \frac{\abs{x_k-y_k}}{\norm{x_k}}$.
  Then $\norm{u}\le\frac{\theta}{2K}$. For every $n=1,\dots,m$, we
  have
  \begin{displaymath}
    \Bigabs{\sum_{k=1}^n\alpha_ky_k}
    \le\Bigabs{\sum_{k=1}^n\alpha_kx_k}
    +\sum_{k=1}^n\abs{\alpha_k}\cdot\abs{x_k-y_k}
    \le\Bigabs{\sum_{k=1}^n\alpha_kx_k}+2K\norm{x}u.
  \end{displaymath}
  Therefore,
  \begin{displaymath}
    \bigvee_{n=1}^m\Bigabs{\sum_{k=1}^n\alpha_ky_k}
    \le\bigvee_{n=1}^m\Bigabs{\sum_{k=1}^n\alpha_kx_k}
    +2K\norm{x}u,
  \end{displaymath}
  which yields, after an application of the bibasis inequality
  for~$(x_k)$,
  \begin{displaymath}
    \biggnorm{\bigvee_{n=1}^m\Bigabs{\sum_{k=1}^n\alpha_ky_k}}
    \le M\norm{x}+2K\norm{x}\norm{u}
    \le\frac{M+\theta}{1-\theta}\norm{y}.
  \end{displaymath}
  Therefore, $(y_k)$ satisfies the bibasis inequality and the
  conclusion follows.
\end{proof}

\begin{corollary}\label{subspace}
  Let $(x_k)$ be a bibasic sequence in a Banach lattice~$X$. Then
  every closed infinite-dimensional subspace of $[x_k]$ contains a
  bibasic sequence.
\end{corollary}

\begin{proof}
  Let $Y$ be a closed infinite-dimensional subspace of~$[x_k]$. Using
  Bessaga-Pe{\l}\-czy{\'n}\-ski's selection principle (see, e.g.,
  Proposition~1.a.11 in~\cite{Lindenstrauss:77}), one can find a basic
  sequence $(y_k)$ in $Y$ and a block sequence $(u_k)$ of $(x_k)$ such
  that $\norm{y_k-u_k}\to 0$ sufficiently fast, so that $(y_k)$ is bibasic by
  Theorem~\ref{biPSP} (note that $(u_k)$ is bibasic by
  Corollary~\ref{blocks}).
\end{proof}

\begin{remark}
  Since disjoint sequences are bibasic, it is clear that every Banach
  lattice contains a bibasic sequence. It is open whether every closed
  infinite dimensional subspace of a Banach lattice contains a bibasic
  sequence. By Corollary~\ref{subspace}, this is the case when $X$
  itself has a bibasis.  We will come back to this problem in the
  final section of the paper.
\end{remark}

\begin{example}
  \emph{Bibasic sequences are not stable under duality.} Let $X=c_0$
  and $x_k=\sum_{n=1}^ke_n$. Being a basis of~$c_0$, $(x_k)$ is
  a bibasis by~\Cref{AM-basic-bib}. Its coordinate functionals
  satisfy $x_k^*=e_k^*-e_{k+1}^*$. As in Examples~\ref{dual-summ}
  and~\ref{nonbib-l1}, $(x_k^*)$ fails to be bibasic.
\end{example}

\section{Stability of bibasic sequences under operators}

Suppose that $T\colon X\to Y$ is an isomorphic embedding between
Banach lattices. Then, clearly, $T$ maps basic sequences in $X$ to basic
sequences in~$Y$. What additional requirements should one impose on
$T$ to ensure that $T$ maps bibasic sequences to bibasic sequences?
\Cref{decomp} suggests that one look at operators that preserve
uniform convergence, or at least turn uniform convergence into order
convergence. Inspired by this, we characterize order bounded operators
in a way that mimics the bibasis theorem.




\begin{theorem}\label{uconv-oconv-ob}
  Let $T\colon X\to Y$ be a linear operator between Archimedean vector
  lattices. TFAE:
  \begin{enumerate}
  \item\label{uconv-oconv-ob-ob} $T$ is order bounded;
  \item\label{uconv-oconv-ob-uu} $x_\alpha\goesu 0$ implies
    $Tx_\alpha\goesu 0$ for all nets $(x_\alpha)$ in~$X$;
  \item\label{uconv-oconv-ob-uo} $x_\alpha\goesu 0$ implies
    $Tx_\alpha\goeso 0$ for all nets $(x_\alpha)$ in~$X$;
  \item\label{uconv-oconv-ob-obdd} $x_\alpha\goesu 0$ implies
    that $(Tx_\alpha)$ has an order bounded tail for all nets $(x_\alpha)$ in~$X$.
  \end{enumerate}
\end{theorem}

\begin{proof}
  \eqref{uconv-oconv-ob-ob}$\Rightarrow$\eqref{uconv-oconv-ob-uu}
  Suppose that $T$ is order bounded and
  $x_\alpha\goesu 0$. Then there exists $e\in X_+$ such that for every
  $\varepsilon>0$ there exists $\alpha_0$ such that
  $\abs{x_\alpha}\le\varepsilon e$ whenever $\alpha\ge\alpha_0$. Find
  $a\in Y_+$ with $T[-e,e]\subseteq[-a,a]$. Then for
  $\varepsilon$ and $\alpha_0$ as above, we have
  $\abs{Tx_\alpha}\le\varepsilon a$ for all $\alpha\ge\alpha_0$. This
  shows that $Tx_\alpha\goesu 0$.

  \eqref{uconv-oconv-ob-uu}$\Rightarrow$\eqref{uconv-oconv-ob-uo}$\Rightarrow$\eqref{uconv-oconv-ob-obdd}  is
 trivial.

 \eqref{uconv-oconv-ob-obdd}$\Rightarrow$\eqref{uconv-oconv-ob-ob} Take
 $e\in X_+$; it suffices to show that $T[0,e]$ is order bounded in~$Y$. Let
  \begin{math}
    \Lambda=\bigl\{(n,x)\mid n\in\mathbb N,\ x\in[0,e]\bigr\}
  \end{math}
  ordered lexicographically: $(n,x)\le(m,y)$ whenever $n< m$ or $n=m$
  and $x\le y$. Clearly, $\Lambda$ is a directed set. Consider the
  following net in $X$ indexed by $\Lambda$:
  $v_{(n,x)}=\frac{1}{n}x$. It follows from
  $0\le v_{(n,x)}\le\frac1n e$ that $v_{(n,x)}\goesu 0$. By
  assumption, the net $(Tv_{(n,x)})$ has an order bounded tail, i.e.,
  there exist $n_0\in\mathbb N$, $x_0\in[0,e]$, and $u\in Y_+$ such
  that $Tv_{(n,x)}\in[-u,u]$ whenever $(n,x)\ge(n_0,x_0)$. In
  particular, $\frac{1}{n_0+1}Tx=Tv_{(n_0+1,x)}\in[-u,u]$ for all
  $x\in[0,e]$. This shows that $T[0,e]$ is order bounded in~$Y$.
\end{proof}

\begin{remark}\label{uc-bdd}
  This theorem yields a simple proof of the classical fact that every
  order bounded (and, in particular, every positive) operator from a
  Banach lattice to a normed lattice is norm continuous. Indeed, let
  $T\colon X\to Y$ be such an operator. Suppose that $x_n\goesnorm 0$
  in~$X$; we need to show that $Tx_n\goesnorm 0$ in~$Y$. Suppose not,
  then, after passing to a subsequence, we can find $\varepsilon>0$
  such that $\norm{Tx_n}>\varepsilon$ for all~$n$. Since
  $x_n\goesnorm 0$, passing to a further subsequence, we may assume
  that $x_n\goesu 0$. Theorem~\ref{uconv-oconv-ob} yields that
  $Tx_n\goesu 0$ and, therefore, $Tx_n\goesnorm 0$, which contradicts
  $\norm{Tx_n}>\varepsilon$ for all~$n$.
  %
\end{remark}

We are going to show next that the sequential analogues of the
conditions \eqref{uconv-oconv-ob-uu}, \eqref{uconv-oconv-ob-uo}, and
\eqref{uconv-oconv-ob-obdd} in Theorem~\ref{uconv-oconv-ob} are also
equivalent, even though they do not imply order boundedness. Moreover,
we can consider operators defined on a subspace $Z$ of $X$ instead of
all of~$X$. For a sequence $(x_n)$ in~$Z$, the notation $x_n\goesu 0$
means that the sequence converges to zero uniformly in~$X$.

\begin{proposition}\label{uconv-oconv-seq}
  Let $X$ and $Y$ be two Archimedean vector lattices, $Z$ a subspace
  of~$X$, and $T\colon Z\to Y$ a linear operator. TFAE:
  \begin{enumerate}
  \item\label{uconv-oconv-seq-uu} $x_n\goesu 0$ implies
    $Tx_n\goesu 0$ for all sequences $(x_n)$ in~$Z$;
  \item\label{uconv-oconv-seq-uo} $x_n\goesu 0$ implies
    $Tx_n\goeso 0$ for all sequences $(x_n)$ in~$Z$;
  \item\label{uconv-oconv-seq-obdd} $x_n\goesu 0$ implies
    $(Tx_n)$ is order bounded for all sequences $(x_n)$ in~$Z$.
  \end{enumerate}
\end{proposition}

\begin{proof}
  \eqref{uconv-oconv-seq-uu}$\Rightarrow$\eqref{uconv-oconv-seq-uo}$\Rightarrow$\eqref{uconv-oconv-seq-obdd}
  is trivial. Suppose that \eqref{uconv-oconv-seq-obdd} holds; let
  $(x_n)$ be a sequence in $Z$ such that $x_n\goesu 0$ in~$X$. It is easy
  to see that there is a sequence $(\lambda_n)$ in $\mathbb R_+$ such
  that $\lambda_n\uparrow\infty$ and $\lambda_nx_n\goesu 0$. By
  assumption, the sequence $T(\lambda_nx_n)$ is order bounded. Let
  $a\in Y_+$ such that $T(\lambda_nx_n)\in[-a,a]$ for every~$n$. It
  follows that $\abs{Tx_n}\le\frac{1}{\lambda_n}a$, so that
  $Tx_n\goesu 0$.
\end{proof}

An operator which satisfies the equivalent conditions of
Proposition~\ref{uconv-oconv-seq} will be called \term{sequentially
  uniformly continuous}. By Theorem~\ref{uconv-oconv-ob}, every order
bounded operator is sequentially uniformly
continuous. Remark~\ref{uc-bdd} shows that every sequentially
uniformly continuous operator from a closed subspace of a Banach
lattice to a Banach lattice is norm continuous.

\begin{proposition}
   Let $T\colon Z\to Y$ be an operator from a closed subspace $Z$ of a
  Banach lattice $X$ to a Banach lattice~$Y$. Then $T$ is
  sequentially uniformly continuous iff the sequence
  $\bigl(\bigvee_{k=1}^n\abs{Tx_k}\bigr)_n$ is norm bounded whenever
  $x_k\goesu 0$.
\end{proposition}

\begin{proof}
  The forward implication is an immediate corollary of
  Proposition~\ref{uconv-oconv-seq}. To prove the converse, suppose
  that $x_k\goesu 0$ implies that the sequence
  $\bigl(\bigvee_{k=1}^n\abs{Tx_k}\bigr)_n$ is norm bounded. Then this
  sequence is also norm bounded in~$Y^{**}$. Since $Y^{**}$ is
  monotonically complete by, e.g., Proposition~2.4.19(ii)
  in~\cite{Meyer-Nieberg:91}, this sequence is order bounded
  in~$Y^{**}$, hence $(Tx_k)$ is order bounded in~$Y^{**}$. By
  Proposition~\ref{uconv-oconv-seq}, $T$ is sequentially uniformly
  continuous as an operator from $Z$ to~$Y^{**}$. Hence, if
  $x_k\goesu 0$ then $Tx_k\goesu 0$ in~$Y^{**}$. Since $Y$ is a closed
  sublattice in~$Y^{**}$, it follows from
  Proposition~\ref{uconv-stable} that $Tx_k\goesu 0$
  in~$Y$. Therefore, $T\colon Z\to Y$ is sequentially uniformly
  continuous.
\end{proof}

\begin{example}
  \emph{A sequentially uniformly continuous operator which fails to be
    order bounded.} Let $T\colon c\to c_0$, defined by
  \begin{displaymath}
    T\colon(a_1,a_2,\dots)\mapsto(a_\infty,a_\infty-a_1,a_\infty-a_2,\dots),
    \text{ where }a_\infty=\lim_na_n.
  \end{displaymath}
  Note that
  \begin{math}
    T\colon(0,\dots,0,1,1,\dots)\mapsto(1,\dots,1,0,0,\dots),
  \end{math}
  so that $T[0,\one]$ is not order bounded in~$c_0$, hence $T$ is not
  order bounded. On the other hand, suppose that $x_n\goesu 0$
  in~$c$. Then $x_n\goesnorm 0$ in~$c$. Since $T$ is norm bounded, it
  follows that $Tx_n\goesnorm 0$ in~$c_0$. It is easy to see that this
  yields $Tx_n\goesu 0$ in~$c_0$.
\end{example}

\begin{proposition}\label{sequcont}
  Let $(X_k)$ be a bidecomposition in a Banach lattice~$X$, let
  $T\colon[X_k]\to Y$ be a sequentially uniformly continuous norm
  isomorphic embedding into a Banach lattice~$Y$. Then the sequence
  $(TX_k)$ is a bidecomposition.
\end{proposition}

\begin{proof}
  Since $T$ is a norm isomorphic embedding, $(TX_k)$ is a Schauder
  decomposition in~$Y$. Suppose that $y\in[TX_k]$ in~$Y$. Then
  $y=\s Tx_k$ for some $x_k\in X_k$. It follows that
  $y=Tx$, where $x=\s x_k$ in~$X$. Since $(X_k)$ is
  a bidecomposition, we have $x=\tus x_k$. By assumption,
  $y=\tus Tx_k$.
\end{proof}

\begin{corollary}
  Let $T\colon X\to Y$ be an order bounded norm isomorphic embedding
  between Banach lattices. Then $T$ maps bibasic sequences to bibasic
  sequences.
\end{corollary}

In Theorem~5.1 of~\cite{Gumenchuk:15} the authors observe that
$L_1[0,1]$ admits no bibasis; moreover, $L_1[0,1]$ does not even embed
by means of a lattice embedding $T$ with $\sigma$-order continuous
inverse map into a $\sigma$-order continuous Banach lattice with a
bibasis. A careful analysis of the proof, together with the preceding
results of the current paper, reveals that the assumptions may
be considerably relaxed:

\begin{theorem}
  There is no norm isomorphic embedding $T$ from $L_1[0,1]$ to a
  Banach lattice such that $\Range T$ is contained in the closed
  linear span of a bibasic sequence (or even of a bi-FDD) and
  $T^{-1}$, viewed as an operator from $\Range T$ to $L_1[0,1]$, is
  sequentially uniformly continuous.
\end{theorem}

\section{Unconditional and permutable decompositions}

Recall that a Schauder decomposition $(X_k)$ is unconditional if every
convergent series $\sum_{k=1}^\infty x_k$ with $x_k\in X_k$ converges
unconditionally; see \cite[p.~534]{Singer:81} or
\cite[1.g]{Lindenstrauss:77} for properties of unconditional
decompositions. By an \term{unconditional bidecomposition} we mean an
unconditional Schauder decomposition which is also a bidecomposition.

\begin{proposition}\label{unc-bid}
  A sequence of closed non-zero subspaces $(X_k)$ of a Banach lattice
  $X$ is an unconditional bidecomposition of $[X_k]$ iff there exists
  a constant $L$ such that
  \begin{equation}\label{unc-bid-ineq}
    \sup\limits_{\varepsilon_k=\pm 1}\biggnorm{\bigvee_{n=1}^m\Bigabs{
        \sum_{k=1}^n \varepsilon_kx_k}}
    \le L\Bignorm{\sum_{k=1}^mx_k}
  \end{equation}
  for any $m\in \mathbb{N}$ and any $x_1\in X_1,\dots,x_m\in X_m$.
\end{proposition}

\begin{proof}
  Suppose that $(X_k)$ is a bidecomposition. Then
  \begin{displaymath}
    \biggnorm{\bigvee_{n=1}^m\Bigabs{
        \sum_{k=1}^n \varepsilon_kx_k}}
    \le M\Bignorm{\sum_{k=1}^m\varepsilon_k x_k}
    \le MK_u\Bignorm{\sum_{k=1}^mx_k},
  \end{displaymath}
  where $M$ is the bidecomposition constant and $K_u$ is the unconditional
  constant of~$(X_k)$. Conversely, suppose that \eqref{unc-bid-ineq}
  is satisfied for any $x_k\in X_k$ as $k=1,\dots,m$. This clearly
  implies the bidecomposition inequality~\eqref{bidec-ineq}, hence $(X_k)$ is a
  bidecomposition. Furthermore,
  \begin{displaymath}
    \sup\limits_{\varepsilon_k=\pm 1}\Bignorm{
     \sum_{k=1}^m \varepsilon_kx_k}
    \le\sup\limits_{\varepsilon_k=\pm 1}\biggnorm{\bigvee_{n=1}^m\Bigabs{
     \sum_{k=1}^n \varepsilon_kx_k}}
    \le L\Bignorm{\sum_{k=1}^mx_k},
  \end{displaymath}
  hence $(X_k)$ is unconditional.
\end{proof}

It is easy to see that, analogously to the theory of unconditional
decompositions, the supremum over all choices of signs in
Proposition~\ref{unc-bid} may be replaced with the supremum over all
choices of $\delta_k\in \{0,1\}$ or the supremum over all $ \beta_k$
with $\abs{\beta_k}\le 1$.  However, this analogy breaks if we
consider permutations of the index set. Recall that a basic sequence
is unconditional iff every permutation of it is again a basic
sequence. Clearly, every permutation of an unconditional bibasic
sequence is an unconditional basic sequence. However, the following
example shows that the bibasis property may be lost after a
permutation.

\begin{example}
  \textit{A permutation of an unconditional bibasis need not be a
    bibasis:} Indeed, the Haar system $(h_k)$ in $L_p[0,1]$ is an
  unconditional bibasis when $1<p<\infty$. However, there is a
  function in $L_\infty[0,1]$ whose Haar series diverges a.e.,\ and,
  therefore, cannot converge in order, after a rearrangement of the
  series; see~\cite[p.~96]{Kashin:89}.  This shows that the
  corresponding rearrangement of $(h_k)$ fails to be a bibasis.
\end{example}

The preceding example motivates the following definition. A
bidecomposition is said to be \term{permutable} if every permutation
of it is a bidecomposition. Similarly, a bibasic sequence is
\term{permutable} if every permutation of it is bibasic. The preceding
example shows that the Haar system in $L_p[0,1]$ ($1<p<\infty$) fails
to be permutable. It is easy to see that permutability implies
unconditionality. 

If $(x_k)$ is an unconditional basic sequence, the supremum of the
basis constants over all permutations of $(x_k)$ is finite. We
establish a similar result for permutable decompositions, even though
the uniform boundedness principle is not applicable in this context.

\begin{theorem}\label{perm-const}
  Let $(X_k)$ be a permutable bidecomposition. The supremum of the
  bibasis constants over all permutations of $(X_k)$ is finite.
\end{theorem}

\begin{proof}
  For the sake of contradiction, assume that the supremum is
  infinite. We claim, then, that the supremum of the bidecomposition
  constants over all permutations of $(X_k)_{k\ge 2}$ is also
  infinite. Suppose not.  Then there is a constant $M$ such that for
  any distinct $k_1,\dots,k_m$ with $k_i\ne 1$ and
  $x_{k_i}\in X_{k_i}$ for $i=1,\dots,m$ we have
  \begin{math}
    \Bignorm{\bigvee_{n=1}^m\bigabs{\sum_{i=1}^n x_{k_i}}}
      \le M\bignorm{\sum_{i=1}^mx_{k_i}}.
  \end{math}
  So if we take any distinct indices $k_1,\dots,k_m$ with $k_{i_0}=1$
  for some $i_0\in \{1,\dots,m\}$, and $x_{k_i}\in X_{k_i}$ as
  $i=1,\dots,m$ then
  \begin{multline*}
    \biggnorm{\bigvee_{n=1}^m\Bigabs{\sum_{i=1}^n x_{k_i}}}
    \le\biggnorm{\abs{x_{k_{i_0}}}+\bigvee_{n=1}^m
       \Bigabs{\sum_{i\in \{1,\dots,n\}\setminus \{i_0\}} x_{k_i}}}\\
    \le\norm{x_{k_{i_0}}}+M\Bignorm{\sum_{i\in\{1,\dots,m\}\setminus \{i_0\}}x_{k_i}}
    \le (K_u+MK_u)\Bignorm{\sum_{i=1}^mx_{k_i}},
  \end{multline*}
  where $K_u$ is the unconditional constant of~$(X_k)$. This
  contradicts the assumption, and, therefore, proves the claim.
  Proceeding inductively, we deduce that the supremum of the bidecomposition
  constants over all permutations of $(X_k)_{k\ge N}$ is infinite for
  every~$N$.

  Hence, we can find distinct indices
  $k^1_1,\dots,k^1_{m_1}$ and vectors $x_{k^1_1}\in X_{k^1_1}$, \dots,
  $x_{k^1_{m_1}}\in X_{k^1_{m_1}}$ such that
  \begin{displaymath}
    \biggnorm{\bigvee_{n=1}^{m_1}\Bigabs{\sum_{i=1}^n x_{k^1_i}}}
    \ge\Bignorm{\sum_{i=1}^{m_1}x_{k^1_i}}.
  \end{displaymath}
  Let $N_1=\max\{k^1_1,\dots,k^1_{m_1}\}$. Applying the previous
  paragraph to $(X_k)_{k>N_1}$, we find distinct
  $k^2_1,\dots,k^2_{m_2}>N_1$ and $x_{k^2_1}\in X_{k^2_1}$, \dots,
  $x_{k^2_{m_2}}\in X_{k^2_{m_2}}$ such that
  \begin{displaymath}
    \biggnorm{\bigvee_{n=1}^{m_2}\Bigabs{\sum_{i=1}^n x_{k^2_i}}}
    \ge 2\Bignorm{\sum_{i=1}^{m_2}x_{k^2_i}}.
  \end{displaymath}
  We then repeat the process in the obvious way; the elements of
  $\mathbb N$ that are missed we enumerate as
  $\l_1,l_2,\dots$. The sequence
  $k^1_1,\dots,k^1_{m_1},l_1,k^2_1,\dots,k^2_{m_2},l_2,\dots$ is a
  permutation of~$\mathbb N$, say,~$\sigma$, and it is clear that
  under this permutation, $(X_{\sigma(k)})$ fails the bidecomposition
  inequality and, hence, is not a bidecomposition.
\end{proof}

\begin{corollary}\label{perm-eq}
  Let $(X_k)$ be a sequence of closed non-zero subspaces of a Banach
  lattice~$X$. TFAE:
  \begin{enumerate}
  \item\label{perm-eq-perm} $(X_k)$ is a permutable bidecomposition of~$[X_k]$;
  \item\label{perm-eq-ineq} There is a constant $M$ such that for
    any sequence $(x_k)$ with $x_k\in X_k$ and any distinct indices
    $k_1,\dots,k_m$, we have
  \begin{displaymath}
    \biggnorm{\bigvee_{n=1}^m\Bigabs{\sum_{i=1}^n x_{k_i}}}
    \le M\Bignorm{\sum_{i=1}^mx_{k_i}}.
  \end{displaymath}
  \end{enumerate}
\end{corollary}

\begin{proof}
  \eqref{perm-eq-perm}$\Rightarrow$\eqref{perm-eq-ineq} Let $M$ be the supremum of the bidecomposition constants, guaranteed to be finite by Theorem~\ref{perm-const}. Choose a permutation $\sigma$ with
  $\sigma(i)=k_i$ as $i=1,\dots,m$. The bidecomposition inequality
  for $(X_{\sigma(k)})$ yields \eqref{perm-eq-ineq}.
  
  \eqref{perm-eq-ineq}$\Rightarrow$\eqref{perm-eq-perm} 
  Let $\sigma$ be a permutation. Applying \eqref{perm-eq-ineq} with
  $m\in\mathbb N$ and $k_i=\sigma(i)$ as $i=1,\dots,m$, we conclude
  that  $(X_{\sigma(k)})$ satisfies the bidecomposition inequality, hence is a
  bidecomposition, which yields \eqref{perm-eq-perm}.
\end{proof}

\section{Absolute decompositions}

Let $(X_k)$ be a permutable bidecomposition, and let $P_n^\sigma$
denote the $n$-th canonical projection associated to the
permutation~$\sigma$. By \Cref{decomp}\eqref{decomp-obdd}, for each
$x\in [X_k]$ there exists $u^\sigma\in X_+$ such that
$\abs{P_n^\sigma x}\le u^\sigma$ for all~$n$. Motivated by
\Cref{perm-const}, it is natural to wonder if one can choose
$u^\sigma$ independent of~$\sigma$. It turns out that one cannot; to
do so one must further modify the basis inequality. This leads to the
following definition, which is of interest in its own right.

\begin{definition}
  Let $(X_k)$ be a sequence of closed non-zero subspaces of a Banach
  lattice~$X$. We say that $(X_k)$ is an \term{absolute decomposition}
  of $[X_k]$ if there exists a constant $A\ge 1$ such that for any
  $m\in \mathbb{N}$ and any $x_1\in X_1,\dots,x_m\in X_m$,
 \begin{equation}\label{abs-ineq}
   \Bignorm{\sum_{k=1}^m\abs{x_k}}
   \le A\Bignorm{\sum_{k=1}^mx_k}.
 \end{equation}
 In particular, a sequence $(x_k)$ is \term{absolute} if there exists a
 constant $A\ge 1$ such that for any $m\in\mathbb N$ and any
 $\alpha_1,\dots,\alpha_m$ we have
 \begin{math}
   \bignorm{\sum_{k=1}^m\abs{\alpha_k x_k}}
   \le A\bignorm{\sum_{k=1}^m\alpha_k x_k}.
 \end{math}
\end{definition}

It is clear that every absolute decomposition is a Schauder
decomposition. Every permutation of an absolute decomposition
satisfies the bidecomposition inequality; it follows that every
absolute decomposition is permutable and, therefore, unconditional.
Moreover, one can easily check that the absolute property is stable
under permutation.

We next prove an absolute version of \Cref{decomp}. Recall that for
any sequence $(x_k)$ it follows from
\begin{math}
  \bignorm{\sum_{k=n}^mx_k}\le\bignorm{\sum_{k=n}^m\abs{x_k}}
\end{math}
whenever $n\le m$ that if $\sum_{k=1}^\infty\abs{x_k}$ converges then
$\sum_{k=1}^\infty x_k$ converges. 

\begin{theorem}\label{abs}
  Let $X$ be a Banach lattice and $(X_k)\subseteq X$ a Schauder
  decomposition of~$[X_k]$. TFAE:
  \begin{enumerate}
  \item\label{abs-const} $(X_k)$ is absolute;
  \item\label{abs-abs} The convergence of $\sum_{k=1}^\infty x_k$
    implies the convergence of $\sum_{k=1}^\infty\abs{x_k}$ for every
    sequence $(x_k)$ with $x_k\in X_k$;
  \item\label{abs-obdd} If $\sum_{k=1}^\infty x_k$ converges then
    the sequence of partial sums
    $\bigl(\sum_{k=1}^m\abs{x_k}\bigr)_m$ is order bounded for all
    $x_k\in X_k$;
  \item\label{abs-norm-bdd} If $\sum_{k=1}^\infty x_k$ converges then
    the sequence of partial sums
    $\bigl(\sum_{k=1}^m\abs{x_k}\bigr)_m$ is norm bounded for all
    $x_k\in X_k$;
  \end{enumerate}
\end{theorem}

\begin{proof}
  \eqref{abs-const}$\Rightarrow$\eqref{abs-abs}: Let $(x_k)$ be such
  that $x_k\in X_k$ for every $k$ and $\sum_{k=1}^\infty x_k$
  converges. Then for each $n\le m$ the absolute
  inequality~\eqref{abs-ineq} yields that
  \begin{displaymath}
    \Bignorm{\sum_{k=n}^m\abs{x_k}}
    \le A\Bignorm{\sum_{k=n}^mx_k}.
   \end{displaymath}
   Hence, $\bigl(\sum_{k=1}^m\abs{x_k}\bigr)_m$ is Cauchy, so that
   $\sum_{k=1}^\infty\abs{x_k}$ converges.



  \eqref{abs-abs}$\Rightarrow$\eqref{abs-obdd}$\Rightarrow$\eqref{abs-norm-bdd}
  trivially.

  \eqref{abs-norm-bdd}$\Rightarrow$\eqref{abs-const}:
  For every $m$ and $x=\sum_{k=1}^\infty x_k$, where $x_k\in X_k$ for
  every~$k$, we define
  \begin{displaymath}
    \varphi_m(x)=\sum_{k=1}^m\abs{x_k}
    =\abs{P_1x}+\abs{P_2x-P_1x}+\dots+\abs{P_mx-P_{m-1}x}.
  \end{displaymath}
  It is clear that $\varphi_m$ is continuous, so that the set
  \begin{displaymath}
    F_i=\Bigl\{x\in[X_k]\mid
    \forall m\in\mathbb N\ \Bignorm{\sum_{k=1}^m\abs{x_k}}\le i\Bigr\}
    =\bigcap_{m=1}^\infty\bigl\{x\in[X_k]\mid \bignorm{\varphi_m(x)}\le i\bigr\}
  \end{displaymath}
  is closed for every $i\in\mathbb N$. The rest of the proof is
  analogous to that of
  \eqref{decomp-nbdd}$\Rightarrow$\eqref{decomp-ineq} in Theorem~\ref{decomp}.
\end{proof}

\begin{remark}
  Since the sequence of the partial sums
  $\bigl(\sum_{k=1}^m\abs{x_k}\bigr)$ in~\eqref{abs-abs} is
  increasing, one easily sees that the convergence is, in fact,
  uniform.
\end{remark}

Recall the following standard fact:

\begin{lemma}\label{2unc}
  For any vectors $x_1,\dots,x_m$ in an Archimedean vector lattice we
  have
  \begin{enumerate}
  \item\label{2unc-signs}
  \begin{math}
    \sum_{k=1}^m\abs{x_k}
    =\sup\sum_{k=1}^m\varepsilon_k x_k
    =\sup\Bigabs{\sum_{k=1}^m\varepsilon_k x_k},
  \end{math}
  where the supremum is taken over all choices of signs $\varepsilon_k=\pm
  1$;
  \item\label{2unc-perm}
  \begin{math}
    \sum_{k=1}^m\abs{x_k}\le 2\sup\bigabs{\sum_{i=1}^nx_{k_i}},
  \end{math}
  where the supremum is taken over all choices of $n\le m$ and $1\le
  k_1<\dots<k_n\le m$. 
  \end{enumerate}  
\end{lemma}

\begin{proof}
  These statements hold for real numbers and, therefore, for
  elements of every Archimedean vector lattice.
\end{proof}

\Cref{abs} immediately yields the characterization of absolute
decompositions that motivated this section:

\begin{proposition}\label{abs-perm}
  Let $X$ be a Banach lattice and $(X_k)$ a bidecomposition in~$X$. TFAE:
  \begin{enumerate}
  \item $(X_k)$ is absolute;
  \item for each $x\in [X_k]$ there exists $u\ge 0$ such that
    $\abs{P_n^\sigma x}\le u$ for all $n\in\mathbb N$ and all
    permutations~$\sigma$.
  \end{enumerate}   
\end{proposition}

\begin{proof}
  Suppose $(X_k)$ is absolute and take
  $x=\sum_{k=1}^\infty x_k\in [X_k]$.  It is clear that
  $u:=\sum_{k=1}^\infty\abs{x_k}$ is as required;
  \Cref{decomp}\eqref{decomp-obdd} is one way to see that $(X_k)$ is
  permutable.

  To prove the converse, let $x\in [X_k]$ and find $u$ as in the
  statement. Lemma~\ref{2unc}\eqref{2unc-perm} yields that for
  each~$m$,
  \begin{math}
    \sum_{k=1}^m\abs{x_k}\le 2u.
  \end{math}
  Thus, $(X_k)$ is absolute by Theorem~\ref{abs}\eqref{abs-obdd}.
\end{proof}

There are several other natural ways to motivate the concepts of an
absolute decomposition and an absolute sequence. In view of
Lemma~\ref{2unc}\eqref{2unc-signs}, the
absolute inequality~\eqref{abs-ineq} may be viewed as the
unconditional inequality
\begin{math}
  \sup\limits_{\varepsilon_k=\pm 1}
  \bignorm{\sum_{k=1}^m\varepsilon_kx_k}
  \le K_u\bignorm{\sum_{k=1}^mx_k}
\end{math}
with the supremum pulled inside the norm:
\begin{math}
  \bignorm{\sup\limits_{\varepsilon_k=\pm 1}\sum_{k=1}^m\varepsilon_kx_k}
  \le A\bignorm{\sum_{k=1}^mx_k}.
\end{math}
On the other hand, it is easy to see that a normalized basic sequence
$(x_k)$ in a Banach space is equivalent to the unit vector basis of
$\ell_1$ iff the convergence of $\sum_{k=1}^\infty\alpha_kx_k$ is
equivalent to the convergence of
$\sum_{k=1}^\infty\norm{\alpha_kx_k}$. Replacing norm with modulus, we
obtain the definition of an absolute sequence.

We know that ``absolute'' $\Rightarrow$ ``permutable'' $\Rightarrow$
``unconditional''. We will now list several cases where the three
concepts are equivalent.

\begin{remark}
  For positive basic sequences, being absolute is equivalent to being
  unconditional. Indeed, let $(x_k)$ be a positive unconditional basic
  sequence.  Fix $\alpha_1,\dots,\alpha_m$. Then
  \begin{displaymath}
    \Bignorm{\sum_{k=1}^m\abs{\alpha_kx_k}}
        =\Bignorm{\sum_{k=1}^m\abs{\alpha_k}x_k}
        \le K_u\Bignorm{\sum_{k=1}^m \alpha_kx_k},
   \end{displaymath}
   where $K_u$ is the unconditional constant of~$(x_k)$.   
\end{remark}

\begin{proposition}
  A Schauder decomposition in an AM-space is absolute iff it is
  unconditional. In particular, a basic sequence in an AM-space is
  absolute iff it is unconditional.
\end{proposition}

\begin{proof}
  Let $(X_k)$ be an unconditional Schauder
  decomposition with unconditional constant~$K_u$. Then for
  $x_1\in X_1,\dots,x_m\in X_m$ Lemma~\ref{2unc}\eqref{2unc-signs}
  yields
  \begin{displaymath}
    \Bignorm{\sum_{k=1}^m\abs{x_k}}
    =\biggnorm{\sup\limits_{\varepsilon_k=\pm 1}
      \Bigabs{\sum_{k=1}^m\varepsilon_kx_k}}
    =\sup\limits_{\varepsilon_k=\pm 1}
      \Bignorm{\sum_{k=1}^m\varepsilon_kx_k}
    \le K_u\Bignorm{\sum_{k=1}^m x_k}.
  \end{displaymath}

\end{proof}

\begin{proposition}\label{l1-abs}
  Every basic sequence in a Banach lattice which is equivalent to the
  unit vector basis of $\ell_1$ is absolute. In particular, a basis of
  $\ell_1$ is absolute iff it is unconditional.
\end{proposition}

\begin{proof}
  Suppose that $(x_k)$ is a basic sequence in a Banach lattice such
  that $(x_k)$ is $M$-equivalent to the unit vector basis $(e_k)$
  of~$\ell_1$. In particular, $(x_k)$ is seminormalized, so that there
  exists $C>0$ such that $\norm{x_k}\le C$ for every~$k$. For every
  $\alpha_1,\dots,\alpha_m$, we have
  \begin{displaymath}
    \Bignorm{\sum_{k=1}^m\abs{\alpha_kx_k}}
    \le C\sum_{k=1}^m\abs{\alpha_k}
    =C\Bignorm{\sum_{k=1}^m\alpha_ke_k}_{\ell_1}
    \le CM\Bignorm{\sum_{k=1}^m\alpha_kx_k}.
  \end{displaymath}
  Up to equivalence, $\ell_1$ has only one normalized unconditional
  basis; see~\cite[Proposition~2.b.9]{Lindenstrauss:77}. Hence, given
  any unconditional basis $(x_k)$ of~$\ell_1$,
  $\bigl(\frac{x_k}{\norm{x_k}}\bigr)$ is equivalent to the unit
  vector basis of $\ell_1$ and, therefore, is absolute. It follows
  that $(x_k)$ is absolute.
\end{proof}

\begin{question}
  Does there exist a Banach lattice with a bibasis but no conditional
  bibasis?
\end{question}

\begin{proposition}
  Suppose that $(x_k)$ is an absolute basic sequence. If
  $\bigl(\abs{x_k}\bigr)$ is a basic sequence then it is dominated by~$(x_k)$.
\end{proposition}

\begin{proof}
  Fix $\alpha_1,\dots,\alpha_m$. Let $I_+=\{k\mid \alpha_k\ge 0\}$ and
  $I_-=\{k\mid \alpha_k< 0\}$. Then
  \begin{multline*}
    \Bignorm{\sum_{k=1}^m\alpha_k\abs{x_k}}
    =\Bignorm{\sum_{k\in I_+}\abs{\alpha_kx_k}-
      \sum_{k\in I_-}\abs{\alpha_kx_k}}
    \le\Bignorm{\sum_{k\in I_+}\abs{\alpha_kx_k}}
    +\Bignorm{\sum_{k\in I_-}\abs{\alpha_kx_k}}\\
    \le A\Bignorm{\sum_{k\in I_+}\alpha_kx_k}
    +A\Bignorm{\sum_{k\in I_-}\alpha_kx_k}
    \le 2AK_u\Bignorm{\sum_{k=1}^m\alpha_kx_k},
  \end{multline*}
  where $A$ is the absolute constant, and $K_u$ is the
  unconditional constant of~$(x_k)$. 
\end{proof}

\begin{example}
  In general, even if $(x_k)$ is an absolute basis, the sequence
  $\bigl(\abs{x_k}\bigr)$ need not be basic. For example, take
  $X=\ell_p$ ($1\le p<\infty$), and put $x_1=e_1+e_2$, $x_2=e_1-e_2$,
  and $x_k=e_k$ whenever $k>2$.
\end{example}

\begin{example}
  \emph{An absolute sequence $(x_k)$ such that the sequence
    $(\abs{x_k})$ is conditional basic, hence not equivalent
    to~$(x_k)$.} Let $X=\ell_\infty$ and let $x_k$ be the $k$-th row
  of the following infinite matrix: {\small
  \begin{displaymath}
  \renewcommand{\tabcolsep}{1pt}
    \begin{tabular}{|cc|cccc|cccccccc|cccccccccccccccc|cccccccccc}
    1&-1&1&1&-1&-1&1&1&1&1&-1&-1&-1&-1&
     1&1&1&1&1&1&1&1&-1&-1&-1&-1&-1&-1&-1&-1&1&1&1&1&1&1&1&1&1&\dots\\
    0&0&1&-1&1&-1&1&1&-1&-1&1&1&-1&-1&
      1&1&1&1&-1&-1&-1&-1&1&1&1&1&-1&-1&-1&-1&1&1&1&1&1&1&1&1&-1&\dots\\
    0&0&0&0&0&0&1&-1&1&-1&1&-1&1&-1&
      1&1&-1&-1&1&1&-1&-1&1&1&-1&-1&1&1&-1&-1&1&1&1&1&-1&-1&-1&-1&1&\dots\\
    0&0&0&0&0&0&0&0&0&0&0&0&0&0&
      1&-1&1&-1&1&-1&1&-1&1&-1&1&-1&1&-1&1&-1&1&1&-1&-1&1&1&-1&-1&1&\dots\\
    0&0&0&0&0&0&0&0&0&0&0&0&0&0&
      0&0&0&0&0&0&0&0&0&0&0&0&0&0&0&0&1&-1&1&-1&1&-1&1&-1&1&\dots\\
      \vdots&&\vdots&&&&\vdots&&&&&&&&\vdots&&&&&&&&&&&&&&&&\vdots&&&&&&
    \end{tabular}
  \end{displaymath}
  }%
  Each $x_k$ is made up of ``blocks'', where initial blocks are zeros
  and further blocks are discrete finite Rademacher vectors. For any
  $m\in\mathbb N$ and any $\alpha_1,\dots,\alpha_m$, there is a column
  in the matrix whose first $m$ entries match the signs of
  $\alpha_1,\dots,\alpha_m$. This yields
  \begin{math}
    \bignorm{\sum_{k=1}^m\alpha_kx_k}=\sum_{k=1}^m\abs{\alpha_k}
    =\bignorm{\sum_{k=1}^m\abs{\alpha_kx_k}}.
  \end{math}
  It follows that $(x_k)$ is a 1-unconditional 1-absolute basic
  sequence equivalent to the unit vector basis of~$\ell_1$.  On the
  other hand, it can be easily verified that $\bigl(\abs{x_k}\bigr)$
  is a conditional basic sequence.
\end{example}

\begin{remark}
  In~\cite{Bu:90}, the authors define a \term{lattice decomposition}
  of a Banach lattice $X$ as a Schauder decomposition $(X_k)$ such
  that $X=[X_k]$ and for every~$k$, the operator $Q_k=P_k-P_{k-1}$,
  which is a projection onto~$X_k$, is a lattice homomorphism (we take
  $P_0=0$). More generally, let $(X_k)$ be a Schauder decomposition of
  $X$ with $X=[X_k]$ such that $Q_k\ge 0$ for every~$k$. Such a
  decomposition is absolute with absolute constant $A=1$. Indeed, let
  $x=\sum_{k=1}^mx_k$, where $x_k\in X_k$. Then
  \begin{displaymath}
    \sum_{k=1}^m\abs{x_k}=\sum_{k=1}^m\abs{Q_kx}
    \le\sum_{k=1}^mQ_k\abs{x}\le\sum_{k=1}^\infty Q_k\abs{x}
    =\abs{x},
  \end{displaymath}
  so that
  \begin{math}
    \bignorm{\sum_{k=1}^m\abs{x_k}}\le
    \bignorm{\sum_{k=1}^mx_k}.
  \end{math}

  It can be easily verified that if $(x_k)$ is a basis such that
  $Q_k\ge 0$ for every $k$ then $X$ is atomic and $(x_k)$ is a
  disjoint sequence of atoms.
  %
  %
  %
  We don't have a good understanding of the structure of those Banach
  lattices $X$ which admit FDDs $(X_k)$ with $Q_k$ positive for each~$k$. In
  particular, does $X$ have atoms? Is there a disjoint sequence such
  that each $X_k$ is the span of a block of this sequence? Notice that
  if the $Q_k$'s are lattice homomorphisms then both these questions have
  positive answers.
\end{remark}


\section{Permutable and absolute sequences in $L_p$ spaces}

Suppose $1\le p<\infty$. We mentioned in Example~\ref{Haar} that the
Haar basis $(h_k)$ of $L_p[0,1]$ is a bibasis iff $p>1$. In this case,
it follows that every block sequence of it is bibasic. In particular,
if $p>1$ then the Rademacher sequence~$(r_k)$, being a block sequence
of~$(h_k)$, is a bibasic sequence. The latter statement remains valid
for $p=1$:

\begin{proposition}\label{Rad-p-bib}
  Let $1\le p<\infty$. The Rademacher sequence $(r_k)$ is a
  bibasic sequence in $L_p[0,1]$. Furthermore, it is permutable but
  not absolute.
\end{proposition}

\begin{proof}
  Let $(x_k)$ be a permutation of the Rademacher sequence. Fix scalars
  $\alpha_1,\dots,\alpha_m$ and let $f_n=\sum_{k=1}^n\alpha_kx_k$ as
  $n=1,\dots,m$. It is easy to see that $(f_n)_{n=1}^m$ is a
  martingale with difference sequence $d_k=\alpha_kx_k$. The associated square function is
  $S(f)=\bigl(\sum_{k=1}^m\alpha_k^2\bigr)^{\frac12}\one$. Applying
  Burkholder-Gundy-Davis inequality~\eqref{BGD} followed by
  Khintchine's inequality, we get
  \begin{displaymath}
    \biggnorm{\bigvee_{n=1}^m\Bigabs{\sum_{k=1}^n\alpha_kx_k}}_{L_p}
    \le C\Bigl(\sum_{k=1}^m\alpha_k^2\Bigr)^{\frac12}
    \le C'\Bignorm{\sum_{k=1}^m\alpha_kx_k}_{L_p}.
  \end{displaymath}
  Hence $(x_k)$ is a bibasic sequence, which yields that $(r_k)$ is a
  permutable bibasic sequence.

  Furthermore, it follows from $\abs{r_k}=\one$ that
  \begin{math}
    \bignorm{\sum_{k=1}^m\abs{\alpha_kr_k}}_{L_p}=\sum_{k=1}^m\abs{\alpha_k},
  \end{math}
  while Khintchine's inequality yields that
   \begin{math}
     \bignorm{\sum_{k=1}^m\alpha_kr_k}_{L_p}
     \sim\Bigl(\sum_{k=1}^m\alpha_k^2\Bigr)^\frac12.
  \end{math} 
  As these two quantities are not equivalent, we conclude that $(r_k)$
  is not absolute.
\end{proof}

The fact that the Rademacher sequence is bibasic in $L_1[0,1]$ may be
generalized as follows.

\begin{proposition}\label{unc-bib-L1}
  Let $(x_k)$ be a block sequence of the Haar basis~$(h_k)$. If $(x_k)$
  is unconditional then it is bibasic in $L_1[0,1]$.
\end{proposition}

\begin{proof}
  Fix scalars $\alpha_1,\dots,\alpha_m$ and let
  $f_n=\sum_{k=1}^n\alpha_kx_k$ as $n=1,\dots,m$.  Since $(h_k)$ is a
  martingale difference sequence, so is~$(x_k)$, hence $(f_n)$ is a
  martingale. Applying Burkholder-Gundy-Davis inequality~\eqref{BGD},
  we get
  \begin{displaymath}
    \biggnorm{\bigvee_{n=1}^m\Bigabs{\sum_{k=1}^n\alpha_kx_k}}
    \sim\biggnorm{\Bigl(\sum_{k=1}^m\abs{\alpha_kx_k}^2\Bigr)^{\frac12}}.
  \end{displaymath}
  By Khintchine's inequality, there is a constant $C$ such that
  \begin{displaymath}
    \Bigl(\sum_{k=1}^m\abs{\alpha_kx_k}^2\Bigr)^{\frac12}
    \le C\frac{1}{2^m}\sum_{\varepsilon_k=\pm 1}
    \Bigabs{\sum_{k=1}^m\varepsilon_k\alpha_kx_k}.
  \end{displaymath}
  Indeed, this inequality is true for real numbers, hence it remains
  valid for vectors in~$X$. It follows that
  \begin{displaymath}
    \biggnorm{\Bigl(\sum_{k=1}^m\abs{\alpha_kx_k}^2\Bigr)^{\frac12}}
    \le C\frac{1}{2^m}\sum_{\varepsilon_k=\pm 1}
    \Bignorm{\sum_{k=1}^m\varepsilon_k\alpha_kx_k}
    \le CK_u\Bignorm{\sum_{k=1}^m\alpha_kx_k},
  \end{displaymath}
  where $K_u$ is the unconditional constant of~$(x_k)$.
  %
  Therefore, $(x_k)$ is bibasic.
\end{proof}

It is known that every closed infinite-dimensional subspace of
$L_1[0,1]$ contains an unconditional basic sequence; see,
\cite[Corollary~12]{Rosenthal:73} or \cite[p.~38]{Lindenstrauss:79}. 

\begin{corollary}\label{sub-L1-unc-bib}
  Every closed infinite-dimensional subspace of an AL-space contains
  an unconditional bibasic sequence.
\end{corollary}

\begin{proof}
  Let $X$ be a closed infinite-dimensional subspace of an
  AL-space~$L$. WLOG we may take $L=L_1[0,1]$. Indeed, it is easy to
  see that we may assume WLOG that $X$ is separable. Replacing $L$
  with the closed sublattice generated by~$X$, we may assume WLOG that
  $L$ is separable. It is well-known that, up to a lattice isometry,
  $L$ is one of the following: $\ell_1$, $L_1[0,1]$,
  $\ell_1\oplus_1L_1[0,1]$, or $\ell_1^m\oplus_1L_1[0,1]$; see, e.g.,
  \cite{Lacey:76} or Section~2.7 in~\cite{Meyer-Nieberg:91}.  All
  these spaces can be lattice isometrically embedded into $L_1[0,1]$,
  so we may assume that $L=L_1[0,1]$.

  \emph{Case 1: $X$ is non-reflexive.} Since $L_1[0,1]$ is a KB-space,
  $X$ contains no isomorphic copy of~$c_0$. By Theorem~1.c.5
  in~\cite{Lindenstrauss:79}, $X$ contains an isomorphic copy
  of~$\ell_1$, and, therefore, $X$ contains a basic sequence which is
  equivalent to the unit vector basis of~$\ell_1$. By
  Proposition~\ref{l1-abs}, it is absolute; in particular, it is
  unconditional and bibasic.
   
  \emph{Case 2: $X$ is reflexive.}  Fix a normalized unconditional basic
  sequence $(x_k)$ in~$X$. Since $X$ is reflexive, $(x_k)$ is weakly
  null. Passing to a subsequence and using
  Bessaga-Pe{\l}\-czy{\'n}\-ski's selection principle, we find a block
  sequence $(u_k)$ of the Haar basis $(h_k)$ such that
  $\norm{x_k-u_k}\to 0$ sufficiently fast so that $(u_k)$ is
  equivalent to~$(x_k)$. It follows that $(u_k)$ is unconditional and,
  therefore, bibasic by
  Proposition~\ref{unc-bib-L1}. Theorem~\ref{biPSP} now yields that,
  after passing to further subsequences if necessary, $(x_k)$ is bibasic.
\end{proof}

\begin{proposition}
  A normalized basic sequence in an AL-space is absolute
  iff it is equivalent to the unit vector basis of~$\ell_1$.
\end{proposition}

\begin{proof}
  Let $(x_k)$ be a normalized basic sequence in an AL-space. For
  every $\alpha_1,\dots,\alpha_m$, we have
  \begin{displaymath}
  \Bignorm{\sum_{k=1}^m\abs{\alpha_kx_k}}
  =\sum_{k=1}^m\abs{\alpha_k}
  =\Bignorm{\sum_{k=1}^m\alpha_ke_k},
  \end{displaymath}
  where $(e_k)$ is the standard unit vector basis of~$\ell_1$. It
  follows  that $(x_k)$ is absolute iff it is
  equivalent to~$(e_k)$.
\end{proof}

It is also true that every normalized
absolute sequence in $L_p(\mu)$ is equivalent to the unit vector basis
of~$\ell_p$. Our proof is based on the proof of Theorem~2
in~\cite{Johnson:15}. 

\begin{theorem}\label{abs-Lp}
  Every normalized absolute sequence in $L_p(\mu)$ ($1\le p<\infty$)
  is equivalent to the unit vector basis $(e_k)$ of~$\ell_p$.
\end{theorem}

\begin{proof}
  Let $(x_k)$ be a normalized absolute sequence in $L_p(\mu)$.  Fix
  $\alpha_1,\dots,\alpha_m$. Put $f_k=\alpha_kx_k$. Being absolute,
  $(x_k)$ is unconditional, so that
  \begin{math}
    \Bignorm{\sum_{k=1}^mf_k}\sim
        \Bignorm{\sum_{k=1}^m\varepsilon_kf_k}
  \end{math}
  for every choice of signs $\varepsilon_k=\pm 1$. Using Fubini's
  Theorem and Khintchine's inequality, we get
  \begin{multline*}
    \Bignorm{\sum_{k=1}^mf_k}^p\sim
    \Ave\limits_{\varepsilon_k=\pm 1}
    \Bignorm{\sum_{k=1}^m\varepsilon_kf_k}^p
    =\int_{t=0}^1\Bignorm{\sum_{k=1}^mr_k(t)f_k}^pdt\\
    =\int_{t=0}^1\int_\Omega\Bigabs{\sum_{k=1}^mr_k(t)
    f_k(\omega)}^pd\omega\,dt
    \lesssim\int_\Omega\Bigl(\sum_{k=1}^m\abs{f_k(\omega)}^2\Bigr)^{\frac{p}{2}}
    d\omega
    =\Bignorm{\Bigl(\sum_{k=1}^m\abs{f_k}^2\Bigr)^\frac12}^p.
  \end{multline*}
  
  Thus,
  \begin{math}
    \Bignorm{\sum_{k=1}^mf_k}\lesssim
    \Bignorm{\Bigl(\sum_{k=1}^m\abs{f_k}^2\Bigr)^\frac12}.
  \end{math}
  On the other hand, since $(x_k)$ is absolute, we have
  \begin{displaymath}
    \Bignorm{\Bigl(\sum_{k=1}^m\abs{f_k}^2\Bigr)^\frac12}
    \le\Bignorm{\sum_{k=1}^m\abs{f_k}}
    \lesssim\Bignorm{\sum_{k=1}^mf_k},
  \end{displaymath}
  so that
  \begin{math}
    \Bignorm{\sum_{k=1}^mf_k}\sim
    \Bignorm{\Bigl(\sum_{k=1}^m\abs{f_k}^2\Bigr)^\frac12}.
  \end{math}
  
  If $1\le p\le 2$, we have
  \begin{displaymath}
    \Bignorm{\sum_{k=1}^mf_k}\lesssim
    \Bignorm{\Bigl(\sum_{k=1}^m\abs{f_k}^2\Bigr)^\frac12}
    \le\Bignorm{\Bigl(\sum_{k=1}^m\abs{f_k}^p\Bigr)^\frac1p}
    \le\Bignorm{\sum_{k=1}^m\abs{f_k}}
    \lesssim\Bignorm{\sum_{k=1}^mf_k},    
  \end{displaymath}
  so that
  \begin{displaymath}
    \Bignorm{\sum_{k=1}^m\alpha_kx_k}
    =\Bignorm{\sum_{k=1}^mf_k}
    \sim\Bignorm{\Bigl(\sum_{k=1}^m\abs{f_k}^p\Bigr)^\frac1p}
    =\Bigl(\sum_{k=1}^m\abs{\alpha_k}^p\Bigr)^\frac1p.
  \end{displaymath}
  
  Now suppose that $2<p<\infty$. Then
  \begin{displaymath}
    \Bignorm{\sum_{k=1}^mf_k}
    \sim\Bignorm{\Bigl(\sum_{k=1}^m\abs{f_k}^2\Bigr)^\frac12}
    \gtrsim\Bignorm{\Bigl(\sum_{k=1}^m\abs{f_k}^p\Bigr)^\frac1p}
    =\Bigl(\sum_{k=1}^m\abs{\alpha_k}^p\Bigr)^\frac1p.
  \end{displaymath}
  To prove the opposite inequality, let $\theta$ be such that
  \begin{math}
    \frac12=\frac{\theta}{1}+\frac{1-\theta}{p}.
  \end{math}
  Using H\"older's inequality in the form
  \begin{math}
    \bigl(\int f\bigr)^\lambda\bigl(\int g)^\mu
    \ge\int f^\lambda g^\mu
  \end{math}
  where $\lambda,\mu\ge 0$ with $\lambda+\mu=1$, and 
  $f$ and $g$ are two positive integrable functions, we get
  \begin{multline*}
    \Bignorm{\sum_{k=1}^m\abs{f_k}}^{p\theta}\cdot
    \Bignorm{\Bigl(\sum_{k=1}^m\abs{f_k}^p\Bigr)^\frac1p}
      ^{p(1-\theta)}
      =\biggl(\int\Bigl(\sum_{k=1}^m\abs{f_k}\Bigr)^p\biggr)^\theta
      \cdot\biggl(\int\sum_{k=1}^m\abs{f_k}^p\biggr)^{1-\theta}\\
      \ge\int\Bigl(\sum_{k=1}^m\abs{f_k}\Bigr)^{p\theta}\cdot
      \Bigl(\sum_{k=1}^m\abs{f_k}^p\Bigr)^{1-\theta}
      =\int\biggl[\Bigl(\sum_{k=1}^m\abs{f_k}\Bigr)^{2\theta}\cdot
      \Bigl(\sum_{k=1}^m\abs{f_k}^p\Bigr)^{\frac{2(1-\theta)}{p}}\biggr]
      ^{\frac{p}{2}}\\
      \ge\int\Bigl[\sum_{k=1}^m\abs{f_k}^{2\theta}\cdot
      \abs{f_k}^{2(1-\theta)}\Bigr]^{\frac{p}{2}}
      =\int\Bigl[\sum_{k=1}^m\abs{f_k}^2\Bigr]^\frac{p}{2},
  \end{multline*}
  which yields
  \begin{displaymath}
    \Bignorm{\Bigl(\sum_{k=1}^m\abs{f_k}^2\Bigr)^{\frac12}}
    \le\Bignorm{\sum_{k=1}^m\abs{f_k}}^{\theta}\cdot
    \Bignorm{\Bigl(\sum_{k=1}^m\abs{f_k}^p\Bigr)^\frac1p}
    ^{1-\theta}
    \sim\Bignorm{\sum_{k=1}^m f_k}^{\theta}\cdot
    \Bigl(\sum_{k=1}^m\abs{\alpha_k}^p\Bigr)
    ^{\frac{1-\theta}{p}}.
  \end{displaymath}
  It follows that
  \begin{displaymath}
    \Bignorm{\sum_{k=1}^m f_k}\lesssim
    \Bignorm{\sum_{k=1}^m f_k}^{\theta}\cdot
    \Bigl(\sum_{k=1}^m\abs{\alpha_k}^p\Bigr)
    ^{\frac{1-\theta}{p}},
  \end{displaymath}
  and, therefore,
  \begin{math}
    \Bignorm{\sum_{k=1}^m f_k}\lesssim
        \Bigl(\sum_{k=1}^m\abs{\alpha_k}^p\Bigr)
    ^{\frac{1}{p}}.
  \end{math}
\end{proof}

\begin{remark}
  Proposition~\ref{Rad-p-bib} shows that ``absolute'' cannot be
  replaced with ``permutable'' in \Cref{abs-Lp} when $p\neq 2$.
\end{remark}

\begin{example}
  Let $R$ be the subspace spanned by the Rademacher sequence in
  $L_p[0,1]$, $1\le p<\infty$ and $p\ne 2$. Then $R$ contains no
  absolute sequence. Indeed, if $(x_k)$ is a sequence in $R$ which is
  absolute as a sequence in $L_p[0,1]$ then $(x_k)$ is equivalent to
  the unit vector basis of $\ell_p$ by Theorem~\ref{abs-Lp}. However,
  $R$ is isomorphic to~$\ell_2$, hence it contains no isomorphic copy
  of~$\ell_p$.
\end{example}

\begin{example}
  \emph{A permutable bibasic sequence in $\ell_p$ with $p\ne 2$ which
    is not equivalent to the unit vector basis; in particular, it is
    not absolute.}  We construct such a sequence as a discretization
  of the Rademacher sequence.  Fix $1\le p<\infty$ with $p\ne 2$. Let
  $s\in\mathbb N$. There is a natural lattice isometric embedding
  $T\colon\ell_p^{2^s}\to L_p[0,1]$ via $Te_i=f_i/\norm{f_i}$, where
  $f_i=\one_{[\frac{i-1}{2^s},\frac{i}{2^s}]}$. Then $\Range T$
  consists of all dyadic functions of level up to $s$ in
  $L_p[0,1]$. In particular, it contains the first $s$ terms of the
  Rademacher sequence. Let $z_k=T^{-1}r_k$ as $k=1,\dots,s$. Then this
  finite sequence is $C_p$-equivalent to the unit vector basis
  of~$\ell_2^s$. Also, as in Proposition~\ref{Rad-p-bib},
  \begin{equation}\label{disr-Rad-est}
    \biggnorm{\bigvee_{n=1}^{s}\Bigabs{\sum_{k=1}^n\alpha_kz_k}}_{\ell_p^{2^s}}
    =\biggnorm{\bigvee_{n=1}^{s}\Bigabs{\sum_{k=1}^n\alpha_kr_k}}_{L_p[0,1]}
    \le C_p'\Bignorm{\sum_{k=1}^s\alpha_kz_k}_{\ell_p^{2^s}}.
  \end{equation}
  Relabel $z_k$ as~$z^{(s)}_k$.

  Now view $\ell_p$ as
  \begin{math}
    \Bigl(\bigoplus_{s=1}^\infty\ell_p^{2^s}\Bigr)_{\ell_p}.
  \end{math}
  Merge the sequences $(z^{(s)}_k)_{k=1}^s$ into a sequence
  in~$\ell_p$; denote it~$(x_j)$. We claim that $(x_j)$ is bibasic. It
  suffices to verify the bibasis inequality. Let $x$ be a finite
  linear combination of $x_j$'s. We may assume that $x$ is of the form
  \begin{math}
    \sum_{s=1}^m\sum_{k=1}^s\alpha^{(s)}_kz^{(s)}_k.    
  \end{math}
  Since the inner blocks have disjoint supports, the left hand side of
  the bibasis inequality may be written as
  \begin{multline*}
    \biggnorm{\sum_{s=1}^m\bigvee_{n=1}^s\Bigabs
      {\sum_{k=1}^n\alpha^{(s)}_kz^{(s)}_k}}
    =\Biggl(\sum_{s=1}^m\biggnorm{\bigvee_{n=1}^s\Bigabs
      {\sum_{k=1}^n\alpha^{(s)}_kz^{(s)}_k}}^p\Biggr)^{\frac{1}{p}}\\
     \le C_p'\Biggl(\sum_{s=1}^m\Bignorm{
       \sum_{k=1}^s\alpha^{(s)}_kz^{(s)}_k}^p\Biggr)^{\frac{1}{p}}
     =C_p'\norm{x}.
  \end{multline*}
  Hence, $(x_j)$ is bibasic. On the other hand, since
  $(z^{(s)}_k)_{k=1}^s$ is $C_p$-equivalent to the unit vector basis
  of $\ell_2^s$ for every~$s$, the sequence $(x_j)$ is not equivalent
  to the unit vector basis of~$\ell_p$.

  Finally, we sketch the proof that $(x_j)$ is permutable. We will use
  Proposition~\ref{perm-eq}. As in the proof of
  Proposition~\ref{Rad-p-bib}, the estimate in~\eqref{disr-Rad-est}
  remains valid if we permute the sequence
  $\bigl(z^{(s)}_k\bigr)_{k=1}^s$.  Fix indices $k_1,\dots,k_m$ and
  coefficients $\alpha_1,\dots,\alpha_m$. Then
  \begin{displaymath}
    \bigvee_{n=1}^m\Bigabs{\sum_{i=1}^n\alpha_ix_{k_i}}
    =\sum_{s=1}^\infty\bigvee_{n=1}^m\biggabs{
      \sum_{\substack{i=1,\dots,n\\2^s\le k_i<2^{s+1}}}\alpha_ix_{k_i}}
  \end{displaymath}
  where the sum over $s$ has only finitely many non-zero
  terms. Moreover, the terms are pair-wise disjoint, so that, using
  the permuted version of~\eqref{disr-Rad-est}, we get
  \begin{multline*}
    \biggnorm{\bigvee_{n=1}^m\Bigabs{\sum_{i=1}^n\alpha_ix_{k_i}}}
    =\Biggl(\sum_{s=1}^\infty\biggnorm{\bigvee_{n=1}^m\biggabs{
      \sum_{\substack{i=1,\dots,n\\2^s\le k_i<2^{s+1}}}\alpha_ix_{k_i}}}^p
    \Biggr)^{\frac{1}{p}}\\
    \le C_p'\Biggl(\sum_{s=1}^\infty\Bignorm{\sum_{\substack
        {i=1,\dots,m\\2^s\le k_i<2^{s+1}}}\alpha_ix_{k_i}}^p
    \Biggr)^{\frac1p}
    =C_p'\Bignorm{\sum_{i=1}^m\alpha_ix_{k_i}}.
  \end{multline*}
\end{example}

The basic sequence, constructed in the previous example, is clearly not
a basis. Recall that every permutable basis in $\ell_1$ is
unconditional and, therefore, absolute by
Proposition~\ref{l1-abs}. This motivates the following question:

\begin{question}
  Is there a basis of $\ell_p$ $(1<p<\infty$) which is permutable but
  not absolute?
\end{question}

\begin{example}
  \emph{The Walsh sequence and Krengel's operator.} While the preceding
  example deals with a discretization of the Rademacher sequence, in this
  example we consider the Walsh sequence and its discretization. Let
  $(w_k)$ be the Walsh sequence in $L_2[0,1]$ with its standard
  enumeration as in, e.g.,~\cite{Wade:82}. Then $(w_k)$ is an
  orthonormal basis of $L_2[0,1]$. It is shown in \cite{Hunt:70} (see,
  also, \cite[p.~631]{Wade:82}) that there is a constant $M$ such that
  for every $f\in L_2[0,1]$ with Walsh-Fourier expansion
  $f=\sum_{k=0}^\infty\alpha_kw_k$ one has
  \begin{math}
    \Bignorm{\bigvee_{n=0}^\infty\bigabs{\sum_{k=0}^n
        \alpha_kw_k}}\le M\norm{f}.
  \end{math}
  It follows from Bibasis Theorem~\ref{bib} that $(w_k)$ is a
  bibasis.

  It is easy to see that $(w_k)$ is not absolute. Indeed, since it is
  an orthonormal basis of $L_2[0,1]$, we have
  \begin{math}
    \Bignorm{\sum_{k=0}^{n-1}w_k}=\sqrt{n}.
  \end{math}
  On the other hand, it follows from $\abs{w_k}=\one$ that
  \begin{math}
    \Bignorm{\sum_{k=0}^{n-1}\abs{w_k}}=n.
  \end{math}
  Hence, the two expressions are not equivalent.

  The Walsh sequence is closely related to the classical example of an
  operator $T\colon\ell_2\to\ell_2$ which is not order bounded, due to
  Krengel; see, e.g., Example~5.6 in~\cite{Aliprantis:06}. The
  operator is defined as follows.  For each $n=0,1,2,\dots$, define
  the Hadamar matrix $H_n$ as follows: $H_0=(1)$,
  \begin{math}
    H_{n+1}=\bigl[
    \begin{smallmatrix}
      H_n & H_n\\
     H_n & -H_n
    \end{smallmatrix}
    \bigr].
  \end{math}

  Put $T_n=2^{-\frac{n}{2}}H_n$. Viewed as an operator on
  $\ell_2^{2^n}$, $T_n$ is a surjective isometry. We view
  $\ell_2=\bigoplus_{n=0}^\infty\ell_2^{2^n}$ with
  $T=\bigoplus_{n=0}^\infty T_n$.  Then $T$ is a surjective isometry,
  and, therefore, the sequence $(Te_k)$ is an orthonormal basis
  in~$\ell_2$. Is it bibasic? Since $T$ fails to be order bounded (it
  even fails to be sequentially uniformly continuous), we cannot apply
  Theorem~\ref{sequcont}.

  Let $W_n$ be the matrix obtained from $H_n$ be ordering its columns
  by the number of sign changes; $W_n$ is called the Walsh matrix of
  order~$n$. The columns of $W_n$ viewed as a sequence in
  $\ell_2^{2^n}$ correspond to $(w_k)_{k=0}^{2^n-1}$. Replacing $H_n$
  in the construction of $T$ with~$W_n$, we obtain another surjective
  isometry on~$\ell_2$; denote it by~$S$. The sequence $(Se_k)$ is a
  permutation of $(Te_k)$. We claim that $(Se_k)$ is a bibasis. Note
  that this sequence comes in pairwise disjoint blocks. Within the
  $n$-th block, the sequence $(Se_k)$ may be identified with
  $(w_k)_{k=0}^{2^n-1}$, hence it satisfies the bibasis inequality
  with constant $M$ as before. Thus, $(Se_k)$ is a bibasis.

  Since $W_n$ is obtained by permuting columns in~$H_n$, we have
  $H_n=W_nP_n$ for some permutation matrix~$P_n$. Since both $H_n$ and
  $W_n$ are symmetric, we have $H_n=P_n^TW_n$, so that $T=US$, where
  $U$ is a permutation of the standard basis in~$\ell_2$. It follows
  that $Te_k=U(Se_k)$. Since $U$ is clearly a surjective isometry and
  a lattice homomorphism, Theorem~\ref{sequcont} yields that
  $(Te_k)$ is a bibasis.

  Now let $1< p<\infty$ with $p\ne 2$. In this case, the Walsh
  sequence $(w_k)$ forms a conditional basis in $L_p[0,1]$; see, e.g.,
  \cite[p.~6]{Popov:13} or \cite[pp. 23--24]{Muller:05}. It was shown
  in~\cite{Sjolin:69} that $(w_k)$ satisfies the bibasis inequality,
  hence it is a bibasis. As in the preceding paragraphs, one can
  construct a discretized version of $(w_k)$ in~$\ell_p$; it is easy
  to see that it is a conditional bibasis of~$\ell_p$. Our
  investigation leaves open the existence of a conditional bibasis in
  $\ell_1$, $\ell_2$, and $L_2[0,1]$.
\end{example}

\section{Bibasic sequences with unique order expansions}

A sequence $(x_k)$ in a Banach lattice $X$ is said to have
\term{unique order expansions} if $\tos\alpha_kx_k=\tos\beta_kx_k$
implies that $\alpha_k=\beta_k$ for every~$k$. Clearly, this is
equivalent to zero having a unique order expansion. In particular,
$(x_k)$ is an order basis if every vector has a unique order
expansion. It is easy to see that every bibasic sequence in an order
continuous Banach lattice has unique order expansions.

\begin{remark}\label{sigma-sester}
  In the definition of unique order expansions one must choose whether
  to use order or $\sigma$-order-convergence. We will work with
  order convergence; the reader who prefers $\sigma$-order convergence can
  make the appropriate modifications. For bibases, we do not know if
  the choice of order convergence matters:
\end{remark}


\begin{question}
  Suppose $(x_k)$ is a bibasis with unique $\sigma$-order
  expansions. Does $(x_k)$ have unique order expansions?
\end{question}

\begin{example}\label{nonuniq linfty}
  Let $X=c$, the space of all convergent sequences. Let
  $e_0=(1,1,\dots)$. Then $(e_n)_{n\ge 0}$ is a basis and, therefore,
  a bibasis of~$c$. However, $e_0=\tos e_k$, hence $e_0$ has multiple
  order expansions. Notice, in contrast, that the basis
  $(x_k)_{k\ge 1}$ of $c$ with
\begin{math}
      x_k=(0,\dots,0,1,1,1,\dots)
  \end{math}
  has unique
  order expansions.
\end{example}

\begin{example}
  \emph{Uniqueness of order expansions depends on the ambient space.}
  Let $(x_k)$ be the Schauder system in $C[0,1]$. Since $C[0,1]$ is an
  AM-space, $(x_k)$ is a bibasis.  Yet, it fails to have unique order
  expansions by Example~\ref{C01-Sch}. We are going to construct a
  Banach lattice $X$ such that $C[0,1]$ is a closed sublattice of $X$
  and $(x_k)$ has unique order expansions relative to~$X$.

  For a compact Hausdorff space~$K$, we put $c_0(K)$ to be the space of
  real-valued functions $f$ on $K$ such that the set
  $\bigl\{\abs{f}>\varepsilon\bigr\}$ is finite for every
  $\varepsilon>0$. In particular, $c_0(K)$ contains all the functions
  with finite support. One defines $CD_0(K)$ as the space of
  functions of the form $f+g$ where $f\in C(K)$ and $g\in c_0(K)$. It
  is known that $CD_0(K)$ is an AM-space; $C(K)$ is a norm closed
  sublattice of $CD_0(K)$. We refer the reader
  to~\cite{Abramovich:93,Troitsky:04} and references therein for basic
  properties of $CD_0(K)$-spaces.
 
  Put $X=CD_0[0,1]$.  We claim that order expansions of $(x_k)$ 
  with respect to $CD_0[0,1]$ are unique. Suppose that $\tos\alpha_kx_k=0$
  in $CD_0[0,1]$. Let $s_n=\sum_{k=1}^n\alpha_kx_k$. Note that
  $s_n(0)=\alpha_1$ for every~$n$. It follows that
  $\bigabs{\alpha_1\one_{\{0\}}}\le\abs{s_n}\goeso 0$ in $CD_0[0,1]$,
  hence $\alpha_1=0$. Therefore, $s_n(1)=\alpha_2$ and, therefore,
  $\bigabs{\alpha_2\one_{\{1\}}}\le\abs{s_n}$ for every $n\ge 2$. It
  follows from $s_n\goeso 0$ in $CD_0[0,1]$ that $\alpha_2=0$. It
  follows that $s_n(\frac12)=\alpha_3$ and, therefore,
  $\bigabs{\alpha_3\one_{\{\frac12\}}}\le\abs{s_n}$ for all $n\ge
  3$. This yields $\alpha_3=0$. Proceeding inductively, $\alpha_k=0$
  for all~$k$.
\end{example}

While in general the concept of a bibasic sequence with unique order
expansions may depend on the ambient space, here we present an
interesting example where it does not. Recall that every basic
sequence in $c_0$ is bibasic with unique order expansions.

\begin{theorem}
  Let $(x_k)$ be a basic sequence in~$c_0$. Viewed as a sequence
  in~$\ell_\infty$, it is bibasic with unique order expansions.
\end{theorem}

\begin{proof}
  Clearly, $(x_k)$ is bibasic in~$\ell_\infty$.  Suppose that there
  exists a sequence $(\alpha_k)$ of coefficients, not all of them
  zero, such that $\tos\alpha_kx_k=0$ in~$\ell_\infty$. WLOG,
  $\tos x_k=0$; otherwise, pass to the subsequence of those $x_k$'s for
  which $\alpha_k\ne 0$ and replace $x_k$ with $\alpha_kx_k$.

  Put $s_n=\sum_{k=1}^nx_k$. Then $s_n\goeso 0$ in~$\ell_\infty$. In
  particular, $(s_n)$ converges to zero coordinate-wise and $(s_n)$ is
  order bounded in $\ell_\infty$ and, therefore, norm bounded. Since
  $(x_k)$ is basic, the zero vector has no non-trivial norm
  expansions, so that $(s_n)$ does not converge to zero. It follows
  that there exists $\delta>0$ such that $\norm{s_n}>\delta$ for
  infinitely many values of~$n$.

  Fix a sequence $(\varepsilon_m)$ in $(0,\delta/2)$ such that
  $\varepsilon_m\to 0$.  We will use a variant of a ``gliding hump''
  technique to find an ``almost disjoint'' subsequence of~$(s_n)$. Let
  $P_n\colon\ell_\infty\to\ell_\infty$ be the projection onto the
  first $n$ coordinates; let $Q_n=I-P_n$.

  Choose $n_1$ so that $\norm{s_{n_1}}>\delta$. Since $s_{n_1}$ is
  in~$c_0$, there exists $k_1$ such that
  $\norm{Q_{k_1}s_{n_1}}<\varepsilon_1$. Put $v_1=P_{k_1}s_{n_1}$,
  then $\supp v_1\subseteq[1,k_1]$ and
  $\norm{s_{n_1}-v_1}<\varepsilon_1$. Since $(s_n)$ converges to zero
  coordinate-wise, we can find $n_2>n_1$ such that
  $\norm{P_{k_1}s_{n_2}}<\varepsilon_2$ and
  $\norm{s_{n_2}}>\delta$. Since $s_{n_2}\in c_0$, find $k_2$ such
  that $\norm{Q_{k_2}s_{n_2}}<\varepsilon_2$. Put
  $v_2=P_{k_2}Q_{k_1}s_{n_2}$. Then $\supp v_2\subseteq[k_1+1,k_2]$
  and $\norm{s_{n_2}-v_2}<\varepsilon_2$.

  Proceeding inductively, we produce a subsequence $(s_{n_m})$ of
  $(s_n)$ and a sequence $(v_m)$ such that $\norm{s_{n_m}}>\delta$,
  $\norm{s_{n_m}-v_m}<\varepsilon_m$, and $\supp v_m<\supp v_{m+1}$
  for every~$m$. It follows from $\varepsilon_m<\delta/2$ that
  $\norm{v_m}>\delta/2$. Being a disjoint seminormalized sequence
  in~$c_0$, $(v_m)$ is basic and is equivalent to~$(e_m)$. Passing to
  a further subsequence if necessary, we may assume that $(s_{n_m})$
  is also basic and is equivalent to $(v_m)$ and, therefore,
  to~$(e_m)$. Hence, there is an isomorphic embedding
  $T\colon c_0\to c_0$ with $Te_m=s_{n_m}$. Put
  $W=\Range T=[s_{n_m}]$.

  Put $y_1=s_{n_1}$ and $y_m=s_{n_m}-s_{n_{m-1}}$ when $m>1$. Then
  $y_m\in W$ and $y_m=\sum_{k=n_{m-1}+1}^{n_m}x_k$ for every~$m$. It
  follows that $(y_m)$ is a block sequence of~$(x_k)$, hence is a
  basic sequence. It follows from $T^{-1}y_1=e_1$ and
  $T^{-1}y_m=e_m-e_{m-1}$ that the sequence
  $(e_1,e_2-e_1,e_3-e_2,\dots)$ is basic. This is a contradiction
  because this sequence fails the basis inequality. Indeed, for
  every~$m$, we have
  \begin{displaymath}
    e_1+\tfrac{m-1}{m}(e_2-e_1)+\tfrac{m-2}{m}(e_3-e_2)+
    \dots+\tfrac{1}{m}(e_m-e_{m-1})=\tfrac1m(e_1+\dots+e_m),
  \end{displaymath}
  hence this vector has norm~$\frac1m$, while $\norm{e_1}=1$.
\end{proof}

\begin{remark}
  \Cref{nonuniq linfty} shows that one cannot replace $c_0$ with $c$
  in the above theorem.
\end{remark}

\begin{example}\label{not sester}
  \emph{For bibasic sequences, uniqueness of order expansions is not
    always preserved under small perturbations.}  Let $X=c$. Put
  $y_1=(0,1,1,\dots)$ and $y_k=e_k$ as $k\ge 2$. Clearly, $(y_k)$ is
  basic; since $c$ is an AM-space, it follows that $(y_k)$ is
  bibasic. However, it fails to have unique order expansions as
  $y_1=\prescript{{\rm o}}{}{\sum_{k=2}^\infty\,}y_k$. Let
  $x_1=y_1+\varepsilon e_1$ and $x_k=y_k$ when $k\ge 2$. Picking
  $\varepsilon>0$ sufficiently small, $(y_k)$ is a small perturbation
  of~$(x_k)$, yet $(x_k)$ has unique order expansions.  By amplifying
  this example to the $c_0$-sum of infinitely many copies of $c$ with
  the $\varepsilon$-perturbation in the $n$-th copy of $(x_k)$ going
  to zero sufficiently fast, and then re-enumerating the resulting
  sequence, we can produce a normalized bibasic sequence $(x_k)$ with
  unique order expansions such that for every $\delta>0$ one can find
  a bibasic sequence $(y_k)$ such that
  $\sum_{k=1}^\infty\norm{x_k-y_k}<\delta$ and, nevertheless, $(y_k)$
  fails to have unique order expansions.
\end{example}

We next identify a combination of conditions which guarantees
stability under small perturbation. For a bibasic sequence $(x_k)$ in
a Banach lattice~$X$, we write $[x_k]^o$ for the set of all vectors
$x\in X$ which admit an order expansion of the form
$x=\tos\alpha_kx_k$. Since $(x_k)$ is bibasic, it is immediate that
$[x_k]\subseteq [x_k]^o$. We say that a sequence $(x_k)$ is
\term{sester-basic} if it is bibasic, has unique order expansions, and
$[x_k]=[x_k]^o$. Clearly, every basis with unique order expansions is
sester-basic; if $X$ is order continuous then every bibasic sequence
is sester-basic.  The proof of the following proposition is
straightforward.

\begin{proposition}\label{ses}
  For a sequence~$(x_k)$, TFAE:
  \begin{enumerate}
  \item\label{ses-ses} $(x_k)$ is sester-basic;
  \item\label{ses-bib} $(x_k)$ is bibasic and $\s\alpha_kx_k$ converges
    whenever $\tos\alpha_kx_k$ converges;
  \item\label{ses-iff} $(x_k)$ is basic, and $\s\alpha_kx_k$ converges
    iff $\tos\alpha_kx_k$ converges.
  \end{enumerate}
\end{proposition}



  


\begin{proposition}\label{strong pert}
  Let $(x_k)$ be a sester-basic sequence in a Banach lattice $X$ with
  basis constant~$K$; let $(y_k)$ be a sequence in $X$ such that
  \begin{displaymath}
    2K\sum_{k=1}^\infty \frac{\norm{x_k-y_k}}{\norm{x_k}}<1.
  \end{displaymath}
  Then $(y_k)$ is sester-basic.
\end{proposition}

\begin{proof}
  By Principle of Small Perturbations and Theorem~\ref{biPSP}, $(y_k)$
  is bibasic and equivalent to~$(x_k)$. WLOG $(x_k)$ is normalized,
  hence $(y_k)$ is semi-normalized. Suppose $\tos\alpha_ky_k$
  converges; it suffices to show that $\s\alpha_ky_k$ converges. The
  partial sums $\bigl(\sum_{k=1}^n \alpha_ky_k\bigr)_n$ are order
  bounded, hence $(\alpha_k y_k)$ is order and norm bounded, hence
  $(\alpha_k)$ is bounded.  Put
  $z:=\sum_{k=1}^\infty \alpha_k(x_k-y_k)$. This series converges
  absolutely. Lemma~\ref{nconv-subseq} yields that it converges
  uniformly and, therefore, in order. It follows that the series
  $\tos\alpha_kx_k$ converges, which implies the convergence of
  $\s\alpha_kx_k$. Since $(y_k)\sim(x_k)$, the series $\s\alpha_ky_k$
  converges.
\end{proof}

\begin{example}\label{sester-dep-amb}
    \emph{The
    property of being sester-basic depends on the ambient space.}
    The standard unit vector basis $(e_k)$ of $c_0$ is
    sester-basic in $c_0$ but not in $\ell_\infty$ because in
    $\ell_\infty$ we have $c_0=[e_k]\neq[e_k]^o=\ell_\infty$. This
    example also shows that the inclusion $[x_k]\subseteq[x_k]^o$ may
    be proper.
\end{example}

\begin{question}
  It was observed in Example~\ref{C01-Sch} that the Schauder system in
  $C[0,1]$ is not a sester-basis. Does $C[0,1]$ admit a sester-basis?
  More generally, does every Banach lattice with a bibasis admit a
  sester-basis?
\end{question}

\begin{question}
  Does every block sequence of a bibasis with unique order expansions
  again have unique order expansions? If so, is it sester-basic? More
  generally, is every block sequence of a sester-basic sequence again
  sester-basic?
\end{question}

\section{Uo-bibasic sequences}

Recall that a net $(x_\alpha)$ in an Archimedean vector lattice $X$ is
said to \term{uo-converge} to $x$ if
$\abs{x_\alpha-x}\wedge u\goeso 0$ for every $u\ge 0$; in this case,
we write $x_\alpha\goesuo x$. Clearly, order convergence implies
uo-convergence; the two convergences agree for order bounded nets. If
$Y$ is a regular sublattice of $X$ (in particular, if $Y$ is an ideal
of $X$ or if $Y$ is a closed sublattice and $X$ is an order continuous
Banach lattice) and $(x_\alpha)$ is a net in $Y$ then
$x_\alpha\goesuo 0$ in $Y$ iff $x_\alpha\goesuo 0$ in~$X$. For
sequences in $L_p(\mu)$ spaces with $\mu$ semi-finite, uo-convergence
agrees with convergence almost everywhere. We refer the reader
to~\cite{Gao:17} and references therein for background on
uo-convergence; see also~\cite{Papangelou:64,Fremlin:04}.

Motivated by the definition of a bibasic sequence, we say that a
sequence $(x_k)$ in a Banach lattice $X$ is \term{uo-bibasic} if it is basic
and for each $x\in [x_k]$ the sequence
of partial sums of $x$ uo-converges to~$x$. It is clear that every bibasic
sequence is uo-bibasic.

\begin{example}
In an atomic Banach lattice uo-convergence agrees with point-wise convergence, so every basic sequence is uo-bibasic. In particular,
  the class of uo-bibasic sequences is much larger than the class of
  bibasic sequences, even in $\ell_p$ ($p<\infty$). 
\end{example}

\begin{example}
  The Haar basis $(h_k)$ in its standard ordering is a uo-bibasis in
  $L_p[0,1]$ when $1\le p<\infty$.  In the case when $p>1$, it follows
  from the fact that $(h_k)$ is a bibasis; when $p=1$ the statement
  follows from, e.g., Theorem~4 in \cite[p.~68]{Kashin:89}.
\end{example}

In general, we do not know whether the property of being a uo-bibasic
sequence depends on the ambient space. However, if $Y$ is a closed
regular sublattice of $X$ and $(x_k)$ is a sequence in~$Y$, it is
clear that it is uo-bibasic in $Y$ iff it is uo-bibasic in~$X$.

It can be easily verified that a block sequence of a uo-bibasic
sequence is again uo-bibasic.  We next show that uo-bibasic
sequences are stable under small perturbations; the proof is analogous
to that of Theorem~3.2 in~\cite{Gumenchuk:15}.

\begin{proposition}
  Let $(x_k)$ be a uo-bibasic sequence in a Banach lattice $X$ with
  basis constant~$K$; let $(y_k)$ be a sequence in $X$ such that
  $2K\s\frac{\norm{x_k-y_k}}{\norm{x_k}}<1$. Then $(y_k)$ is uo-bibasic.
\end{proposition}

\begin{proof}
  By the usual Principle of Small Perturbations for basic sequences,
  $(y_k)$ is basic and is equivalent to~$(x_k)$. WLOG, $(x_k)$ is
  normalized.  Suppose that $y=\s\alpha_ky_k$. We need to show that
  $y=\tuos\alpha_ky_k$. Since $(y_k)\sim(x_k)$, the series
  $x:=\s\alpha_kx_k$ converges in norm.  Since $(x_k)$ is uo-bibasic,
  we have $x=\tuos\alpha_kx_k$. Since $(x_k)$ is normalized, the
  sequence $(\alpha_k)$ is bounded. It follows that the series
  $\s\abs{\alpha_k}\abs{x_k-y_k}$ converges. Put
  $v_n=\sum_{k=n+1}^\infty\abs{\alpha_k}\abs{x_k-y_k}$. Then
  $v_n\downarrow 0$. It follows from
  \begin{displaymath}
    \Bigabs{y-\sum_{k=1}^n\alpha_ky_k}
    \le\Bigabs{x-\sum_{k=1}^n\alpha_kx_k}
    +v_n\goesuo 0
  \end{displaymath}
  that $y=\tuos\alpha_ky_k$.
\end{proof}

\begin{question}
  Does every closed infinite dimensional subspace of a Banach lattice
  contain a uo-bibasic sequence?
\end{question}

To finish, we answer the above question in a large class of Banach lattices. 
\begin{theorem}\label{uo in subspace}
  Every closed infinite dimensional subspace of an order continuous
  Banach lattice contains an unconditional uo-bibasic sequence.
\end{theorem}

\begin{proof}
  Let $Y$ be a closed infinite dimensional subspace of an order
  continuous Banach lattice~$X$; we will show that $Y$ contains an
  unconditional uo-bibasic sequence. WLOG, $Y$ is separable. Let
  $\overline{S(Y)}$ be the closed sublattice generated by $Y$
  in~$X$. It is easy to see that $\overline{S(Y)}$ is separable and
  regular in~$X$. Therefore, replacing $X$ with $\overline{S(Y)}$, we
  may assume that $X$ is separable. It follows that $X$ has a weak
  unit.  We may then continuously embed $X$ as a norm dense ideal into
  $L_1(\mu)$ for some probability measure~$\mu$; see Theorem~1.b.14
  in~\cite{Lindenstrauss:79}.  Following the proof of
  Proposition~1.c.8 in~\cite{Lindenstrauss:79}, we reduce to the
  following two cases:

  \emph{Case 1:} The norms $\norm{\cdot}_X$ and
  $\norm{\cdot}_{L_1(\mu)}$ are equivalent on~$Y$. In this case we may
  view $Y$ as a closed subspace of $L_1(\mu)$. By
  Corollary~\ref{sub-L1-unc-bib}, $Y$ contains an unconditional basic
  sequence $(y_k)$ which is bibasic in $L_1(\mu)$. It is left to show
  that $(y_k)$ is uo-bibasic in~$X$. Let $y=\s\alpha_ky_k$, where the
  series converges in norm (it does not matter in which norm because
  $\norm{\cdot}_X$ and $\norm{\cdot}_{L_1(\mu)}$ are equivalent on
  $Y$). Since $(y_k)$ is uo--bibasic in $L_1(\mu)$, we have
  $y=\tuos\alpha_ky_k$ in $L_1(\mu)$. Since $X$ is an ideal in
  $L_1(\mu)$, we conclude that $y=\tuos\alpha_ky_k$ in~$X$.

  \emph{Case 2:} There is a sequence $(y_k)$ in $Y$ and a disjoint
  sequence $(x_k)$ in $X$ such that $\norm{y_k}_X=1$ for all $k$ and
  $\norm{y_k-x_k}_X\to 0$. Being disjoint,
  $(x_k)$ is unconditional and bibasic in~$X$. Passing to a
  subsequence, if necessary, and applying the Principle of Small
  Perturbations, we conclude that $(y_k)$ is unconditional and bibasic
  and, therefore, uo-bibasic in~$X$.
\end{proof}
We do not know if every closed infinite dimensional subspace of an
order continuous Banach lattice contains a bibasic sequence. We also
don't know if such subspaces contain permutable uo-bibasic sequences,
i.e., unconditional basic sequences such that every permutation is
uo-bibasic.

\bigskip

{\bf Acknowledgements.}
The authors would like to thank Bill Johnson for valuable discussions.

\end{document}